\documentclass{m2an}

\usepackage{amsmath,amsthm,amssymb,bm}
\usepackage[mathscr]{eucal}
\usepackage{xcolor,comment,algorithm,algorithmic,graphicx}

\newtheorem{thm}{Theorem}[section]

\newtheorem{prop}[thm]{Proposition}

\theoremstyle{definition}

\newtheorem{rem}[thm]{Remark}

\newtheorem{assum}{Assumption}

\DeclareMathOperator*{\argmax}{arg\,max}

\newcommand{\Pad}{{\mathscr P_\mathsf{ad}}}
\newcommand{\Pscr}{{\mathscr P}}
\newcommand{\Uscr}{{\mathscr U}}
\newcommand{\Uad}{{\mathscr U_\mathsf{ad}}}

\newcommand{\Yscr}{{\mathscr Y}}
\newcommand{\Qscr}{{\mathscr Q}}
\newcommand{\Vc}{{V_\circ}}
\newcommand{\Cscrc}{{C_\circ(\overline\Omega)}}

\newcommand{\Vh}{{V^h}}
\newcommand{\Vch}{{V_\circ^h}}
\newcommand{\yrm}{{\mathrm y}}
\newcommand{\qrm}{{\mathrm q}}
\newcommand{\wrm}{{\mathrm w}}
\newcommand{\Wscr}{{\mathscr W}}

\newcommand{\Vl}{{V^{\ell}}} % RB space for y
\newcommand{\Vcl}{{V_\circ^{\ell}}} % RB space for q
\newcommand{\elly}{{\ell_{\mathrm y}}} % RB dimension of y
\newcommand{\ellq}{{\ell_{\mathrm q}}} % RB dimension of q
\newcommand{\Psiy}{{\Psi_{\mathrm y}}} % RB matrix of y
\newcommand{\Psiq}{{\Psi_{\mathrm q}}} % RB matrix of q

\newcommand{\frm}{{\mathrm f}}
\newcommand{\ellf}{{\ell_{\mathrm f}}} % DEIM dimension
\newcommand{\Psif}{{\Psi_{\frm}}} % DEIM matrix
\newcommand{\Prm}{{\mathrm P}}
\newcommand{\PrmT}{{\mathrm P^\top}}
\newcommand{\tildePrmT}{{\tilde{\mathrm P}^\top}}
\newcommand{\hatPsiy}{{\hat\Psi_{\mathrm y}}}
\newcommand{\hatPsiq}{{\hat\Psi_{\mathrm q}}}

\newcommand{\Psiyl}{{\Psi_{\mathrm{y}}^\ell}}
\newcommand{\Psiql}{{\Psi_{\mathrm{q}}^\ell}}
\newcommand{\Psiym}{{\Psi_{\mathrm{y}}^m}}

\newcommand{\Vm}{{V^{m}}}
\newcommand{\Vcm}{{V_\circ^{m}}}

\newcommand{\Srm}{{\mathrm S}}
\newcommand{\Syrm}{{\mathrm S_\yrm}}
\newcommand{\Sqrm}{{\mathrm S_\qrm}}
\newcommand{\Ptr}{{\mathscr P_\mathsf{train}}}
\newcommand{\Pte}{{\mathscr P_\mathsf{test}}}

\newcommand{\bmu}{{\bm\mu}}
\newcommand{\bmua}{{\bm\mu_\mathsf a}}
\newcommand{\muai}{{\mu_{\mathsf a,i}}}

\newcommand{\bmub}{{\bm\mu_\mathsf b}}
\newcommand{\mubi}{{\mu_{\mathsf b,i}}}

\newcommand{\ua}{{u_\mathsf a}}
\newcommand{\ub}{{u_\mathsf b}}
\newcommand{\ya}{{y_\mathsf a}}
\newcommand{\ka}{{\kappa_\mathsf a}}

\begin{document}
%%-----------------------------
%%      the top matter
%%-----------------------------
\title{Adaptive parameter optimization for an elliptic-parabolic system using the reduced-basis method with hierarchical a-posteriori error analysis}

\thanks{The authors got partial financial support within the COMET K2 Competence Centers for Excellent Technologies from the Austrian Federal Ministry for Climate Action, the Austrian Federal Ministry for Digital and Economic Affairs, the Province of Styria (Dept. 12) and the Styrian Business Promotion Agency.}

\thanks{The authors furthermore acknowledge funding by the Deutsche Forschungsgemeinschaft for the project Localized Reduced Basis Methods for PDE-constrained Parameter Optimization under contract VO 1658/6-1.
}

\author{Behzad Azmi}\address{University of Konstanz, Department of Mathematics and Statistics, Konstanz, 78457 Germany, \\ \email{behzad.azmi@uni-konstanz.de\ \&\ andrea.petrocchi@uni-konstanz.de\ \&\ stefan.volkwein@uni-konstanz.de}}
\author{Andrea Petrocchi}\sameaddress{1}
%\address{University of Konstanz, Department of Mathematics and Statistics, Konstanz, 78457 Germany}
\author{Stefan Volkwein}\sameaddress{1}
\date{\today}
\begin{abstract}
    In this paper the authors study a non-linear elliptic-parabolic system, which is motivated by mathematical models for lithium-ion batteries. One state satisfies a parabolic reaction diffusion equation and the other one an elliptic equation. The goal is to determine several scalar parameters in the coupled model in an optimal manner by utilizing a reliable reduced-order approach based on the reduced basis (RB) method. However, the states are coupled through a strongly non-linear function, and this makes the evaluation of online-efficient error estimates difficult. First the well-posedness of the system is proved. Then a Galerkin finite element and RB discretization are described for the coupled system. To certify the RB scheme hierarchical a-posteriori error estimators are utilized in an adaptive trust-region optimization method. Numerical experiments illustrate good approximation properties and efficiencies by using only a relatively small number of reduced basis functions.
\end{abstract}

\maketitle
%%-----------------------------
\section{Introduction}
\label{Sec:1}
%%-----------------------------

The modelling of lithium-ion batteries has received an increasing amount of attention in the recent past. Several companies worldwide are developing such batteries for consumer electronic applications, in particular, for electric-vehicle applications. To achieve the performance and lifetime demands in this area, exact mathematical models of the battery are required. Moreover, the multiple evaluations of the battery model for different parameter settings involve a large amount of time and experimental effort. Here, the derivation of reliable mathematical models and their efficient numerical realization are very important issues in order to reduce  both computational time and cost in the improvement of the performance of batteries.

Mathematical models for lithium-ion batteries describe the evolution of lithium-ion concentration in the different regions of a battery and the electric potentials in the so-called electrolyte and solid phases. We refer to \cite{PVMI10}, where the transport processes are described by a coupled system of partial differential equations (PDEs). The physical and chemical details can be found, e.g., in \cite{LZ11,LZI11}. The equation system models a physico-chemical micro-heterogeneous battery model.  A macro-homogeneous approach is developed in the pioneering work \cite{New73} and forms the basis for further investigations; cf. \cite{DFN93,FDN94,GWDW02,Ram16,SW06}, for instance. Well-posedness is studied, e.g., in \cite{Seg13,WXZ06}.

The goal of the present work is to make a first step in order to extend the theoretical and numerical results obtained in \cite{PSV22} to a more realistic battery model. For that purpose, we consider the following still simplified coupled system of parametrized elliptic-parabolic equations
\begin{subequations}
    \label{0.0:coupled_system_together}
    \begin{align}
        \label{0.0:coupled_system_together_parabolic}
        y_t (t,x) - \mu_1\big(\kappa_1 (x) y_x (t,x)\big)_x - \mu_2 f(y(t,x), q(t,x)) &= 0&&\text{f.a.a. }(t,x)\in Q_T,\\
        \label{0.0:coupled_system_together_elliptic}
        -\mu_3 ( \kappa_2 (x) q_x (t,x) )_x + \mu_4 f(y(t,x), q(t,x))&= 0&&\text{f.a.a. }(t,x)\in Q_T
    \end{align}
\end{subequations}
together with an initial condition for $y$, homogeneous Neumann boundary conditions for $y$ and inhomogeneous mixed boundary conditions for $q$. Throughout we write `f.a.a.' for `for almost all'. In \eqref{0.0:coupled_system_together} the non-linear mapping $f:\mathbb R_+\times\mathbb R\to\mathbb R$ has the specific form
\begin{align*}
    f(\yrm,\qrm):= \sqrt{\mathrm y} \sinh(\mathrm q)\quad\text{for }\yrm\in\mathbb R_\ge:=\{s\in\mathbb R\,\vert\,s\ge0\}\text{ and }\qrm \in \mathbb R,
\end{align*}
which is motivated by mathematical models for lithium-ion batteries. There are parameters $\bmu=(\mu_i)_{1\le i\le 4}\in\mathbb R^4$ in the PDE system which are assumed to be unknown a-priori or which cannot be determined experimentally. Hence, efficient numerical algorithms are needed to estimate these parameters. For this purpose, a parameter identification problem is formulated as a non-linear least squares problem. To speed-up the optimization method a reduced-basis (RB) scheme is used; see,  e.g., \cite{HRS16,QMN16} and \cite{LV13,LXC01,ORS16,VW13} for battery models. In particular, we apply an adaptive trust-region method that does not need any a-priori computation of an RB subspace on an offline phase, but builds the RB approximation online based on computable a-posteriori error estimates; see \cite{BKMOSV22,BMV22,KMOSV21,PSV22,QGVW17}. Since there are no efficient error bounds available for the non-linear system \eqref{0.0:coupled_system_together}, we utilize hierarchical a-posteriori estimates based on \cite{HORU19}.

Summarizing, the new main contributions of the present paper are: (i) proof of existence of a unique weak solution to \eqref{0.0:coupled_system_together}, (ii) extension of the hierarchical error estimation introduced in \cite{HORU19} to a parabolic and non-linear coupled system, (iii) development of a trust-region method for parameter optimization by combining the algorithms in \cite{BKMOSV22,KMOSV21,QGVW17} and the hierarchical error analysis.

The paper is organized as follows: In Section~\ref{Sec:2} we study the well-posedness and local existence in time of \eqref{0.0:coupled_system_together}. The full- and reduced-order discretization are explained in Section~\ref{Sec:3}. The hierarchical a-posteriori error estimator is derived and tested numerically in Section~\ref{Sec:4}. The parameter estimation is considered in Section~\ref{Sec:5}, where also numerical experiments are presented. Secton~\ref{Sec:6} is devoted to a conclusion. 
In Appendix \ref{Sec:A} the proofs of Section~\ref{Sec:2} are given.

%%-----------------------------
\section{The coupled elliptic-parabolic PDE}
\label{Sec:2}
%%-----------------------------

In this section, we introduce our coupled elliptic-parabolic problem and prove that a unique weak solution exists locally in time.

%%-----------------------------
\subsection{The weak formulation}
\label{Sec:2.1}
%%-----------------------------

Let $T>0$ be the (finite) time horizon, $\Omega := (0,L)\subset\mathbb R$ be a space interval and let $Q_T:= (0,T)\times(0,L)$. We consider the following parameter-dependent parabolic-elliptic coupled system for the two state variables $y,q: Q_T \to \mathbb{R}$
\begin{subequations}
\label{2.0:coupled_system_together}
\begin{align}
    \label{2.0:coupled_system_together_parabolic}
    y_t (t,x) - \mu_1\big(\kappa_1 (x) y_x (t,x)\big)_x - \mu_2 f(y(t,x), q(t,x)) &= 0&&\text{f.a.a. }(t,x)\in Q_T,\\
    \label{2.0:coupled_system_together_elliptic}
    -\mu_3 ( \kappa_2 (x) q_x (t,x) )_x + \mu_4 f(y(t,x), q(t,x))&= 0&&\text{f.a.a. }(t,x)\in Q_T
\end{align}
with homogeneous Neumann boundary conditions
\begin{equation}
    \label{2.0:coupled_system_together_parabolic_BC} 
    y_x(t,0) = y_x(t,L) = 0\quad \text{f.a.a. }t\in(0,T),
\end{equation}
inhomogeneous Dirichlet-Neumann mixed boundary conditions
\begin{equation}
    \label{2.0:coupled_system_together_elliptic_BC} 
    q(t,0) = 0\text{ f.a.a. }t\in(0,T), \quad\mu_3 \kappa_2(L) q_x(t,L) = u(t)\text{ f.a.a. }t\in(0,T)
\end{equation}
and initial conditions
\begin{equation}
    \label{2.0:coupled_system_together_IC} 
    y(0,x) =y_\circ(x)\quad \text{f.a.a. }x\in\Omega.
\end{equation}
\end{subequations}
In the following, we will fix some assumptions for \eqref{2.0:coupled_system_together} that are necessary for our existence results.

\begin{assum}
\label{2.0:assumption_1}
	\hfill
	\begin{enumerate}
        \item [1)] For given parameter bounds $\bmua=(\muai)_{1\le i\le 4}$ and $\bmub=(\mubi)_{1\le i\le 4}$ satisfying $0<\bmua\le\bmub$ in $\Pscr=\mathbb R^4$ an admissible parameter vector $\bmu=(\mu_i)_{1\le i\le 4}$ lies in the nonempty, compact and convex set $\Pad=\{\bmu\in\Pscr:\,\bmua\le\bmu\le\bmub\text{ in }\Pscr\} \subset \Pscr$. Here, `$\le$' is understood component-wise.
        \item [2)] The initial condition $y_{\circ}$ belongs to $H^1(\Omega)$ and is positive on $\overline\Omega$, namely $y_\circ(x)\ge \ya$ for all $x\in\overline\Omega$ and a positive constant $\ya$.
        \item [3)] For $\Uscr:=L^2(0,T)$ the set of admissible inputs is $\Uad:=\{u\in\Uscr:\,\ua(t)\le u(t)\le\ub(t)\text{ f.a.a. }t\in[0,T]\}$ with $\ua,\ub\in L^\infty(0,T)$ satisfying $\ua\le\ub$ a.e. in $[0,T]$. In particular, $\Uad\subset L^\infty(0,T)$ holds, and we have $\|u\|_{L^\infty(0,T)}\le c_\Uscr$ for all $u\in\Uscr$ with $c_\Uscr=\max\{\|\ua\|_{L^\infty(0,T)},\|\ub\|_{L^\infty(0,T)}\}$.
        \item [4)] The functions $\kappa_1$, $\kappa_2$ belong to $C^{0,1}(\overline{\Omega})$ with $\kappa_1(x) \ge \ka$ and $\kappa_2(x) \ge \ka$ for all $x \in \overline{\Omega}$ and a positive constant $\ka$.
        \item [5)] The non-linearity is defined as $f(\mathrm y, \mathrm q) := \sqrt{\mathrm y} \sinh(\mathrm q)$ for $\mathrm y\in\mathbb R_{\ge}$ and $\mathrm q \in \mathbb R$.
        \end{enumerate}
\end{assum}

\begin{rem}
    \hfill
    \begin{enumerate}
        \item [1)] Recall that $H^1(\Omega) \hookrightarrow C(\overline\Omega)$ holds (cf. \cite{Eva10}). Thus, Assumption~\ref{2.0:assumption_1}-2) implies $y_\circ\in C(\overline\Omega)$.
        \item [2)] Note that $f$ is not differentiable at $\mathrm y = 0$. In our application, the state $y$ stands for the concentration of lithium-ions in a battery cell. Thus, the situation $\mathrm y=y(t,x)\le0$ does not have any physical meaning. This non-negativity of $\mathrm y$ is needed to evaluate $f(\mathrm y,\mathrm q)$.\hfill$\Diamond$
    \end{enumerate}
\end{rem}

Let $H := L^2(\Omega)$ and $V := H^1(\Omega)$ endowed with the inner products
\begin{align*}
	{\langle\varphi,\phi\rangle}_H=\int_\Omega\varphi(x)\phi(x)\,\mathrm dx\text{ for }\varphi,\phi\in H,\quad{\langle\varphi,\phi\rangle}_V=\int_\Omega\varphi(x)\phi(x)+\varphi'(x)\phi'(x)\,\mathrm dx\text{ for }\varphi,\phi\in V,
\end{align*}
respectively, and the associated induced norms $\|\cdot\|_H=\langle\cdot\,,\cdot\rangle_H^{1/2}$, $\|\cdot\|_V=\langle\cdot\,,\cdot\rangle_V^{1/2}$. We define the Hilbert space
\begin{align*}
	\Vc:=\{\varphi\in V\,|\,\varphi(0)=0\}
 \end{align*}
supplied by the inner product
\begin{align*}
	{\langle\varphi,\phi\rangle}_\Vc:=\int_\Omega\varphi'(x)\phi'(x)\,\mathrm dx\quad\text{for }\varphi,\phi\in\Vc.
\end{align*}
Furthermore, let $\Cscrc= \{\varphi\in C(\overline\Omega):\,\varphi(0)= 0\}$ supplied with the $C(\overline\Omega)$-norm, i.e.,  $\|\varphi\|_{C(\overline\Omega)}=\max_{x\in\overline\Omega}|\varphi(x)|$. Since $V \hookrightarrow C(\overline\Omega)$ holds, there exists an embedding constant $c_{\mathsf e}>0$ (only dependent on the interval $\Omega$) satisfying
\begin{align}
	\label{2.1:c_embedding}
	{\|\varphi\|}_{C(\overline\Omega)} \le c_{\mathsf e}\,{\|\varphi\|}_V\quad\text{for all }\varphi\in V
\end{align}
with $\|\varphi\|_{C(\overline\Omega)}=\max_{x\in\overline\Omega}|\varphi(x)|$. Furthermore, Poincaré's inequality holds on the space $\Vc$ (see \cite[Theorem~7.91]{Sal16}): there exists a constant $c_{\mathsf P} > 0$ such that
\begin{equation}
    \label{2.1:poincare_inequality}
    {\|\varphi\|}_H\le c_{\mathsf P}\,{\|\varphi\|}_\Vc \quad \text{for all }\varphi\in\Vc.
\end{equation}
From \eqref{2.1:c_embedding} and \eqref{2.1:poincare_inequality} we infer that
\begin{equation*}
    {\|\varphi\|}^2_{C(\overline\Omega)}\le c^2_{\mathsf e}\,{\|\varphi\|}^2_V=c^2_{\mathsf e}\left({\|\varphi\|}^2_H+{\|\varphi\|}^2_\Vc\right) \le c^2_{\mathsf e}\left(c_{\mathsf P}^2+1\right) {\|\varphi\|}^2_\Vc.
\end{equation*}
Hence, for $c_\mathsf{eP}=c_{\mathsf e}(c_{\mathsf P}^2+1)^{1/2}>0$ it holds that
\begin{equation}
    \label{2.1:embedding+poincare}
    {\|\varphi\|}_{C(\overline\Omega)}\le c_\mathsf{eP}\,{\|\varphi\|}_\Vc\quad\text{for all }\varphi\in\Vc.
\end{equation}

Next, we define the solution spaces
\begin{align*}
    \Yscr^T:=W(0,T;V,V') \cap C( \overline Q_T)\quad\text{ and }\quad\Qscr^T:=L^\infty(0,T;\Vc)
\end{align*}
endowed by their product topology, where $W(0,T;X,Y):=L^2(0,T;X)\cap H^1(0,T;Y)$; see, e.g., \cite{DL00} for more details. Moreover, for the function $x\mapsto\varphi(t,x)$ f.a.a. $t\in [0,T]$ with $\varphi\in L^2(0,T;V)$, we frequently use the notation $\varphi(t)$.

Throughout we write `a.e.' for `almost everywhere'. Now, the weak formulation of \eqref{2.0:coupled_system_together} is as follows: for $\bmu \in \Pad$ and $u\in\Uad$ find a solution pair $z=(y,q)\in\Yscr^T\times\Qscr^T$ such that
\begin{subequations}
    \label{2.1:weak_all}
    \begin{align}
        \label{2.1:weak_parabolic}
        \frac{\mathrm d}{\mathrm dt}\,{\langle y(t),\varphi^y\rangle}_H+a^1_\bmu (y(t),\varphi^y)+{\langle g^1_\bmu [z(t)],\varphi^y\rangle}_{V',V}&=0&&\text{for all }\varphi^y\in V\text{ and }t\in(0,T]\text{ a.e.},\\
        % \frac{\mathrm d}{\mathrm dt}\,{\langle y(t),\varphi^y\rangle}_H+a^1_\bmu (y(t),\varphi^y)+g^1_\bmu(z(t),\varphi^y)&=0&&\text{for all }\varphi^y\in V\text{ and }t\in(0,T]\text{ a.e.},\\
        \label{2.1:weak_ic}
        y(0)&=y_\circ&&\text{in } H,\\
        \label{2.1:weak_elliptic}
        a^2_\bmu(q(t),\varphi^q)+{\langle g^2_\bmu[z(t)],\varphi^q\rangle}_{V_\circ',\Vc}&={\langle b(t) , \varphi^q\rangle}_{V_\circ',\Vc}&&\text{for all }\varphi^q\in\Vc\text{ and }t\in(0,T]\text{ a.e.,}
    \end{align}
\end{subequations}
where the bilinear forms are defined as
\begin{subequations}
    \label{2.1:bilinear_forms}
    \begin{align}
        a^1_\bmu(\varphi,\phi)&:=\mu_1\int_\Omega\kappa_1 (x)\varphi'(x)\phi'(x)\,\mathrm dx=: \mu_1 \, \hat{a}^1(\varphi,\phi)&&\text{for }\varphi,\phi\in V, \\
        a^2_\bmu(\varphi,\phi)&:=\mu_3 \int_\Omega \kappa_2(x)\varphi'(x)\phi'(x)\,\mathrm dx=: \mu_3 \, \hat{a}^2 (\varphi,\phi)&&\text{for }\varphi,\phi\in\Vc.
    \end{align}
\end{subequations}
The non-linear operators are given by
\begin{subequations}
    \label{2.1:nonlinear_operators}
    \begin{align}
    {\langle g^1_\bmu [z(t)],\varphi\rangle}_{V',V}&:=\mu_2 \int_\Omega (-f(y(t),q(t))) \varphi(x)\,\mathrm dx=:\mu_2 \, \hat g^1(z(t),\varphi)&&\text{for }\varphi\in V, \\
    {\langle g^2_\bmu [z(t)],\varphi\rangle}_{V_\circ',\Vc}&:=\mu_4\int_\Omega f(y(t),q(t)) \varphi(x)\,\mathrm dx\hspace{5mm}=:\mu_4\,\hat g^2(z(t),\varphi)&&\text{for }\varphi\in\Vc,
    \end{align}
\end{subequations}
Finally, the boundary condition appears in the linear operator
\begin{equation*}
    {\langle b(t),\varphi\rangle}_{\Vc',\Vc}=u(t) \varphi(L)\qquad\text{for }\varphi\in\Vc\text{ and }t\in(0,T]\text{ a.e.}
\end{equation*}

\begin{comment}
    \begin{rem}
    If we even have $y_t\in H^1(0,T;H)$ we have
    %
    \begin{align*}
        & \frac{\mathrm d}{\mathrm dt}\,{\langle y(t),\varphi\rangle}_H={\langle y_t(t),\varphi\rangle}_H\quad\text{for all }\varphi\in V\text{ and f.a.a. }t\in[0,T]
    \end{align*}
    %
    \hfill$\Diamond$
    \end{rem}
\end{comment}

%%-----------------------------
\subsection{Well-posedness of the state equation locally in time}
\label{Sec:2.2}
%%-----------------------------

Suppose that $y\in\Yscr^T$ with $y>0$ in $Q_T$ is given. For every $u\in\Uad$ and f.a.a. $t \in [0,T]$ we consider the non-linear elliptic problem
\begin{equation}
    \label{2.2:elliptic_state} \tag{\textbf E}
    -\mu_3\left(\kappa_2 q_x(t)\right)_x+\mu_4f(y(t),q(t))= 0\text{ a.e. in } \Omega,\quad q(t,0)=0,\quad\mu_3\kappa_2(L)q_x(t,L)=u(t).
\end{equation}
A weak solution of \eqref{2.2:elliptic_state} satisfies f.a.a. $t\in[0,T]$
\begin{equation}
    \label{2.2:elliptic_state_weak}
    \tag{$\mathbf E_{\mathsf w}$}
    \mu_3 \int_\Omega \kappa_2(x) q_x(t,x) \varphi'(x) \,\mathrm dx + \mu_4\int_\Omega f(y(t,x),q(t,x))\varphi(x)\mathrm dx= u(t) \varphi(L)\quad\text{for all }\varphi\in\Vc.
\end{equation}
To prove the existence of a solution to \eqref{2.2:elliptic_state_weak} we make use of the Leray-Schauder fixed point theorem (cf., e.g. \cite[p.~189]{Fri64}). For that purpose, the following hypothesis is needed. 

\begin{assum}
    \label{2.2:assumption_Y_M}
    For a given $M>1$ let $y$ belong to the non-empty, closed, bounded, convex set $\Yscr^T_M$ defined as
    \begin{equation}
        \label{2.2:Y_M}
        \Yscr^T_M := \left\{ y\in C(\overline Q_T):\, \frac{1}{M} \le y(t,x) \le M \text{ for all }(t,x)\in\overline Q_T\right\}.
    \end{equation}
\end{assum}

The following result is proved in Appendix~\ref{Sec:A.1}.

\begin{thm}
    \label{2.2:thm_exist_C_est}
    Let Assumptions~\emph{\ref{2.0:assumption_1}} and \emph{\ref{2.2:assumption_Y_M}} hold. Then, for every $u\in\Uad$ there exists a unique solution $q\in\Qscr^T$ of \eqref{2.2:elliptic_state_weak} and a constant $c(M)>0$ (independent of $q$ or $u$ but dependent on $M$) such that 
    \begin{equation}
        \label{2.2:q_est_C}
        {\| q(t) \|}_{C(\overline\Omega)} \le c_\mathsf{eP}\,{\| q(t) \|}_{V_\circ}\le c(M)|u(t)| \le c(M) c_\Uscr\quad\text{for }t\in[0,T]\text{ a.e.,}
    \end{equation}
    where the constant $c_\mathsf{eP}$ has already been introduced in \eqref{2.1:embedding+poincare}.
\end{thm}

In the following, we show the existence of the weak solution for the coupled system \eqref{0.0:coupled_system_together}. The proof is based on the Schauder fixed point theorem and is given in Section~\ref{Sec:A.2} of the appendix.

\begin{thm}
    \label{thm:ExCoupledSystem}
    Let Assumptions~{\em\ref{2.0:assumption_1}} and {\em\ref{2.2:assumption_Y_M}} hold and the constant $M$ in Assumption~{\em\ref{2.2:assumption_Y_M}} be chosen as
    \begin{equation}
        \label{2.3:M}
        M = 2 \, {\|y_\circ\|}_{C(\overline\Omega)} + \frac{2}{\ya}.
    \end{equation}
   Then, there exists a finite time $T_\circ=T_\circ(M)\in(0,T]$ such that \eqref{2.0:coupled_system_together} admits a unique solution pair $z=(y,q)\in \Yscr^{T_\circ}_M \times \Qscr^{T_\circ}$.
\end{thm}

\begin{rem}
    \label{Rem:y0}
    Note that,  due to \eqref{2.3:M} we can see that $y_\circ \le M/2$ and $y_\circ \ge \ya\ge 2/M$ for all $x\in \overline\Omega$.\hfill$\Diamond$
\end{rem}

%%-----------------------------
\section{The discretization}
\label{Sec:3}
%%-----------------------------

Problem \eqref{2.1:weak_all} has to be discretized for its numerical solution. First, we introduce a standard Galerkin approximation, which leads to a high-dimensional non-linear system of ordinary differential equations. Then, we formulate a reduced-order discretization. 

%%-----------------------------
\subsection{The full-order discretization}
\label{Sec:3.1}
%%-----------------------------

The discretization of \eqref{2.1:weak_all} is done in two steps. Suppose that $\{\varphi_i\}_{i=0}^n\subset V$ are linearly independent and that $\{\varphi_i\}_{i=1}^n\subset \Vc$ holds. We define the finite-dimensional subspaces $\Vh=\mathrm{span}\,\{\varphi_0,\ldots,\varphi_n\}\subset V$ and $\Vch=\mathrm{span}\,\{\varphi_1,\ldots,\varphi_n\}\subset \Vc$. Then, by Galerkin projection of equations \eqref{2.1:weak_all} onto $\Vh$ and $\Vch$, our goal is to find $z^h=(y^h,q^h)\in H^1(0,T;\Vh)\times L^\infty(0,T;\Vch)$ solving
\begin{subequations}
    \label{3.1:galerkin_projected_all}
    \begin{align}
        \label{3.1:galerkin_projected_parabolic}
        \frac{\mathrm d}{\mathrm dt}\,{\langle y^h(t),\varphi^y\rangle}_H+a^1_\bmu (y^h(t),\varphi^y)+{\langle g^1_\bmu [z^h(t)],\varphi^y\rangle}_{V',V}&=0&&\text{for all }\varphi^y\in \Vh\text{ and }t\in(0,T]\text{ a.e.},\\
        \label{3.1:galerkin_projected_ic}
        \langle y^h(0)-y^h_\circ,\varphi^y\rangle_H&=0 &&\text{for all } \varphi^y \in \Vh,\\
        \label{3.1:galerkin_projected_elliptic}
        a^2_\bmu(q^h(t),\varphi^q)+{\langle g^2_\bmu[z^h(t)],\varphi^q\rangle}_{V_\circ',\Vc}-{\langle b(t) , \varphi^q\rangle}_{V_\circ',\Vc}&=0&&\text{for all }\varphi^q\in\Vch\text{ and }t\in[0,T]\text{ a.e.},
    \end{align}
\end{subequations}
where $y^h_\circ$ is a projection of $y_\circ$ on $\Vh$ given as
\begin{align*}
    y_\circ^h=\mathrm{argmin}\,\left\{{\|y_\circ-\varphi^h\|}_H\,\vert\,\varphi^h\in\Vh\right\}.    
\end{align*}
Note that
\begin{equation*}
    y^h(t)=\sum_{i=0}^n \yrm_{i}(t) \varphi_i, \qquad q^h(t)=\sum_{i=1}^n \qrm_{i}(t) \varphi_i \qquad\text{for }t\in[0,T]
\end{equation*}
so that \eqref{3.1:galerkin_projected_all} reduces into finding the coefficient vectors $\yrm(t)=(\yrm(t))_{0\le i\le n}$ and $\qrm(t)=(\qrm(t))_{1\le i\le n}$ solving the differential algebraic system
\begin{subequations}
    \label{3.1:semidiscretized_all}
    \begin{align}
        \label{3.1:semidiscretized_all-a}
        \mathrm M_\yrm \dot\yrm(t) +\mu_1\mathrm A_1 \yrm(t) -\mu_2 \mathrm M_\yrm \mathrm f_\yrm (\yrm(t),\qrm(t))&=0&&t\in(0,T]\text{ a.e.},\\
        \mathrm M_\yrm \yrm(0)&=\mathrm y_\circ,&&\\
        \label{3.1:semidiscretized_all-c}
        \mu_3\mathrm A_2 \qrm(t) + \mu_4 \mathrm M_\qrm \mathrm f_\qrm (\yrm(t),\qrm(t))+\mathrm b(t)&=0&&t\in[0,T]\text{ a.e.}
    \end{align}
\end{subequations}
for
\begin{align*}
    &\mathrm{M}_\yrm=((\langle\varphi_j,\varphi_i\rangle_H))\in\mathbb{R}^{(n+1)\times(n+1)}, &&\mathrm{A}_1=((\hat{a}^1(\varphi_j,\varphi_i)))\in\mathbb{R}^{(n+1)\times(n+1)} && \text{for } i,j=0,\dots,n \\
    &\yrm_\circ=(\langle y_\circ , \varphi_i \rangle_H) \in \mathbb R^{n+1} &&&& \text{for } i=0,\dots,n, \\
    &\mathrm{M}_\qrm=((\langle\varphi_j,\varphi_i\rangle_H))\in\mathbb{R}^{n\times n}, &&\mathrm{A}_2=((\hat{a}^2(\varphi_j,\varphi_i)))\in\mathbb{R}^{n\times n} && \text{for } i,j=1,\dots,n, \\
    &\mathrm{b}(t)=(\langle b(t),\varphi_i\rangle_{\Vc',\Vc})\in\mathbb R^n &&&& \text{for } i=1,\dots,n.
\end{align*}
Moreover, $\mathrm f_\yrm(\yrm(t),\qrm(t)) \in \mathbb R^{n+1}$ (resp. $\mathrm f_\qrm(\yrm(t),\qrm(t)) \in \mathbb R^n$) is the coefficient vector satisfying
\begin{align*}
     f(y^h(t,x),q^h(t,x))\approx \left\{
     \begin{aligned}
         &\sum_{j=0}^n\mathrm f_{\yrm,j}(\yrm(t),\qrm(t))\varphi_j(x)&&\text{in \eqref{3.1:semidiscretized_all-a}},\\
         &\sum_{j=1}^n\mathrm f_{\qrm,j}(\yrm(t),\qrm(t))\varphi_j(x)&&\text{in \eqref{3.1:semidiscretized_all-c}}
     \end{aligned}
     \right\}\quad\text{for }x\in\Omega\text{ and }t\in[0,T]\text{ a.e.}
\end{align*}
Note that none of the matrices depend on $\bmu$ due to the affine dependence of the parameters as in \eqref{2.1:bilinear_forms} and \eqref{2.1:nonlinear_operators}.

For solving \eqref{3.1:semidiscretized_all} we apply the implicit Euler method for the time integration (cf., e.g., \cite{Qua17}) on an equidistant time grid $t_k = (k-1) \Delta t$, $k=1,\dots,K$ and $\Delta t = T/(K-1)$. Then, the problem is to find $\{\yrm^k\}_{k=1}^K\subset\mathbb R^{n+1}$ and $\{\qrm^k\}_{k=1}^K\subset\mathbb R^n$ solving
\begin{equation}
    \label{3.1:fully_discretized}
    \begin{aligned}
        \mathrm{M}_\yrm ( \yrm^{k} - \yrm^{k-1} ) + \mu_1 \Delta t \mathrm{A}_1 \yrm^{k} - \mu_2 \Delta t \mathrm{M}_\yrm \mathrm{f}_\yrm (\yrm^{k}, \qrm^{k})&=0&\text{for }k=2,\dots,K, \\
        \mathrm{M}_\yrm \yrm^{1}&=\yrm_{\circ}, \\
        \mu_3 \mathrm{A}_2 \qrm^{k} + \mu_4 \, \mathrm{M}_\qrm \mathrm{f}_\qrm (\yrm^{k}, \qrm^{k}) + \mathrm{b}^{k}&=0&\text{for }k=1,\ldots,K.
    \end{aligned}
\end{equation}
These solutions are approximations of the FE solutions of \eqref{3.1:galerkin_projected_all} at each time step, namely
\begin{equation*}
    y^h(t_k) \approx \sum_{i=0}^n \yrm^k_i \varphi_i\quad\text{and}\quad q^h(t_k) \approx \sum_{i=1}^n \qrm^k_i \varphi_i\quad \text{for } k=1,\dots,K.
\end{equation*}
Finally, to solve this non-linear system, we use Newton's method at each time step $k$, by defining the non-linear function
\begin{equation*}
    \mathrm F_\bmu (\yrm^k, \qrm^k) = \left(\begin{array}{c}
        (\mathrm M_\yrm+\mu_1\Delta t \mathrm A_1)\yrm^k-\mathrm M_\yrm\yrm^{k-1}-\mu_2\Delta t\mathrm M_\yrm\mathrm f(\yrm^k,\qrm^k)\\[1mm]
        \mu_3\mathrm A_2\qrm^k + \mu_4\mathrm M_\qrm\mathrm f(\yrm^k,\qrm^k)+\mathrm b^k
    \end{array}\right) \in \mathbb R^{2n+1}.
\end{equation*}
Then, the algorithm to evaluate the state variables given parameter $\bmu$ is as follows:
\begin{itemize}
    \item Evaluate $\yrm^1$ by projecting the initial value.
    \item Compute $\qrm^1$ solving the decoupled elliptic PDE given $\yrm^1$ to get a consistent initial condition for the state $q$.
    \item For $k = 2, \dots, K$ find the root of the non-linear equation $\mathrm F_\bmu(\yrm^k,\qrm^k) = 0$.
\end{itemize}

%%-----------------------------
\subsection{The reduced-order discretization}
\label{Sec:3.2}
%%-----------------------------

In order to speed-up calculations, we construct reduced-order spaces $\Vl \subset \Vh$ and $\Vcl \subset \Vch$, of dimensions respectively $\elly\ll n$ and $\ellq\ll n$, and a set of basis for each space, $\{\psi^y_1,\dots,\psi^y_\elly\} \subset \Vl$ and $\{\psi^q_1,\dots,\psi^q_\ellq\} \subset \Vcl$. These spaces are evaluated by utilizing the proper orthogonal decomposition (POD) method (cf., e.g., \cite{VK01}). Furthermore, $\Psiy\in\mathbb{R}^{(n+1)\times\elly}$ is the matrix of coordinates of the basis $\{\psi^y_i\}_{i=1}^\elly$ with respect to the basis $\{\varphi_i\}_{i=0}^n$ of $V^{h}$, and similarly $\Psiq\in\mathbb{R}^{n\times\ellq}$ is the matrix of coordinates of the basis $\{\psi^q_i\}_{i=1}^\ellq$ with respect to the basis $\{\varphi_i\}_{i=1}^n$ of $V^{h}_\circ$. Namely,
\begin{equation*}
    \psi^y_j = \sum_{i=0}^{n} (\Psiy)_{ij} \, \varphi_i \text{ for } j=1,\dots,\elly\quad \mathrm{and}\quad\psi^q_j = \sum_{i=1}^{n} (\Psiq)_{ij} \, \varphi_i \text{ for } j=1,\dots,\ellq.
\end{equation*}

As a first step we project all matrices on the reduced spaces so that a reduced-order approximation of \eqref{3.1:fully_discretized} reads as follows: find $\{\hat\yrm^{k}\}_{k=1}^K\subset\mathbb R^\elly$ and $\{\hat\qrm^{k}\}_{k=1}^K\subset\mathbb R^\ellq$ such that
\begin{equation}
    \label{3.2:ROM_pre_DEIM}
    \begin{aligned}
        (\mathrm{M}_\yrm^\ell+\mu_1\Delta t\mathrm{A}_1^\ell)\hat\yrm^{k}-\mathrm{M}_\yrm^\ell \hat\yrm^{k-1}-\mu_2\Delta t\Psi_\yrm^\top\mathrm{M}_\yrm\mathrm{f}_\yrm(\Psiy \hat\yrm^{k},\Psiq\hat\qrm^{k})&=0&\text{for }k=2,\dots,K,\\
        \mathrm{M}_\yrm^\ell \hat\yrm^{1} &= \yrm^\ell_\circ,\\
        \mu_3\mathrm{A}_2^\ell\hat\qrm^{k}+\mu_4\Psi_\qrm^\top\mathrm{M}_\qrm\mathrm{f}_\qrm(\Psiy\hat\yrm^{k},\Psiq\hat\qrm^{k})+\mathrm{b}^{k,\ell}&=0&\text{for }k=1,\dots,K
    \end{aligned}
\end{equation}
with $\mathrm{M}^\ell_\yrm=\Psi_\yrm^\top\mathrm{M}_\yrm\Psiy$, $\mathrm{A}^\ell_1=\Psi_\yrm^\top\mathrm{A}_1\Psiy$, $\mathrm{A}^\ell_2=\Psi_\qrm^\top\mathrm{A}_2\Psiq$, $\mathrm{b}^{k,\ell}=\Psi_\qrm^\top\mathrm{b}^{k}$ for $k=1,\dots,K$, and $\yrm^\ell_\circ=\Psi_\yrm^\top\yrm_\circ$.

System \eqref{3.2:ROM_pre_DEIM} still depends on dimension $n$ through the evaluation of the non-linearities $\mathrm{f}_\yrm$ and $\mathrm f_\qrm$. We can use the empirical interpolation method (EIM) or the discrete empirical interpolation method (DEIM) (cf., e.g., \cite{BMNP04} and \cite{CS10}) to make the evaluation of the system independent of $n$. In our case, we use DEIM where, in a nutshell, the vectors $\mathrm f^{k}_\yrm:=\mathrm{f}_\yrm(\Psiy\hat\yrm^k,\Psiq\hat\qrm^k)$ and $\mathrm f^{k}_\qrm:=\mathrm{f}_\qrm(\Psiy\hat\yrm^k,\Psiq\hat\qrm^k)$ for $k=1,\dots,K$ are projected onto a smaller space of dimension $\ellf\ll n$. For the $\qrm$-system this is done by finding matrices $\Psif\in\mathbb R^{n\times\ellf}$, and $\Prm\in\mathbb R^{n\times\ellf}$ such that
\begin{equation*}
    \mathrm{f}^k\approx\Psif(\PrmT\Psif)^{-1}\PrmT\mathrm{f}^k,
\end{equation*}
where the matrix $\Psif$ is given by the POD method using the snapshots $\mathrm f^k$ for $k=1,\dots,K$. The matrix $\Prm$ is, on the other hand, composed only by zeros and ones such that $\PrmT\mathrm v\in\mathbb R^\ellf$ contains only selected rows of $\mathrm v\in\mathbb R^n$. For more information, see \cite{CS10}.

\begin{rem}
    For the $\yrm$-system (of dimension $n+1$) we set
    \begin{align*}
        \tilde\Psi_\frm = \left[\begin{array}{c}0\\\Psif\end{array}\right]\in\mathbb{R}^{(n+1)\times\ellf},\quad
        \tilde\Prm = \left[\begin{array}{c}0\\\Prm\end{array}\right]\in\mathbb{R}^{(n+1)\times\ellf},
    \end{align*}
    and get $\tilde\frm^k\approx\tilde\Psi_\frm(\tilde\Prm^\top\tilde\Psi_\frm)^{-1}\tilde\Prm^\top\tilde\frm^k\in\mathbb{R}^{n+1}$.\hfill$\Diamond$
\end{rem}

Since the evaluation of $\frm^k$ is done component-wise, we can write
\begin{equation*}
    \PrmT\mathrm{f}^k=\PrmT\mathrm{f}(\Psiy\hat\yrm^k,\Psiq\hat\qrm^k)=\mathrm{f}(\tildePrmT\Psiy\hat\yrm^k,\PrmT\Psiq\hat\qrm^k)=:\mathrm{f}(\hatPsiy\hat\yrm^k,\hatPsiq\hat\qrm^k)
\end{equation*}
for $\hatPsiy=\tildePrmT\Psiy\in\mathbb R^{\ellf\times\elly}$ and $\hatPsiq=\PrmT\Psiq \in \mathbb R^{\ellf\times\ellq}$.
Then, for 
\begin{equation*}
    \mathrm{G}_\yrm=\Psi_\yrm^\top\mathrm{M}_\yrm\tilde\Psi_\frm(\tilde\Prm^\top\tilde\Psi_\frm)^{-1}\in\mathbb{R}^{\elly\times\ellf}
    \quad\text{and}\quad
    \mathrm{G}_\qrm=\Psi_\qrm^\top\mathrm{M}_\qrm\Psif(\PrmT\Psif)^{-1}\in\mathbb{R}^{\ellq\times\ellf}
\end{equation*}
system \eqref{3.2:ROM_pre_DEIM} is approximated by
\begin{equation}
    \label{3.2:ROM}
    \begin{aligned}
        (\mathrm{M}_\yrm^\ell+\mu_1\Delta t\mathrm{A}_1^\ell)\hat\yrm^{k}-\mathrm{M}_\yrm^\ell\hat\yrm^{k-1}-\mu_2\Delta t\textrm{G}_\yrm\mathrm{f}(\hat\Psi_\yrm\hat\yrm^{k},\hat\Psi_\qrm\hat\qrm^{k}) &=0&\text{for }k=2,\dots,K,\\
        \mathrm{M}_\yrm^\ell \hat\yrm^{1} &= \yrm^\ell_\circ,\\
        \mu_3\mathrm{A}_2^\ell\hat\qrm^{k}+\mu_4\mathrm{G}_\qrm\mathrm{f}(\hat\Psi_\yrm\hat\yrm^{k},\hat\Psi_\qrm\hat\qrm^{k})+\mathrm{b}^{k,\ell}&=0&\text{for }k=1,\dots,K,\\
    \end{aligned}
\end{equation}
which is finally independent of $n$. In the following, we call system \eqref{3.2:ROM} reduced-order (RO) model. The RO solutions $\{\hat\yrm^{k}\}_{k=1}^K\subset\mathbb{R}^\elly$ and $\{\hat\qrm^k\}_{k=1}^K\subset\mathbb{R}^\ellq$ are interpreted as a reduced-order approximations for $\{\yrm^k\}_{k=1}^K\subset\mathbb{R}^{n+1}$ and $\{\qrm^k\}_{k=1}^K\subset\mathbb{R}^n$, namely
\begin{equation*}
    \yrm^k\approx{\yrm}^{k,\ell}:=\Psiy\hat\yrm^{k}
    \qquad\text{and}\qquad
    \qrm^k\approx{\qrm}^{k,\ell}:=\Psiq\hat\qrm^{k}.
\end{equation*}

%%-----------------------------
\section{Hierarchical a-posteriori error for the state equation}
\label{Sec:4}
%%-----------------------------

The accuracy of the reduced-order solution is controlled by hierarchical error estimates. Here we extend the approach in \cite{HORU19} to the time-dependent and non-linear coupled system \eqref{2.0:coupled_system_together}.

%%-----------------------------
\subsection{The error estimator}
\label{Sec:4.1}
%%-----------------------------

From now on, we add the dependence on the parameter $\bmu$. Namely, $\{\yrm^k(\bmu)\}_{k=1}^K$ and $\{\qrm^k(\bmu)\}_{k=1}^K$ are the FE solutions (solving \eqref{3.1:fully_discretized}), while $\{\hat\yrm^k(\bmu)\}_{k=1}^K$ and $\{\hat\qrm^k(\bmu)\}_{k=1}^K$ are the ROM solutions (solving problem \eqref{3.2:ROM}) with parameter $\bmu\in\Pad$.

\textit{A-posteriori estimates} are needed to control the error of the reduced-order approximation without knowing the full-order solution (see, e.g., \cite{BKMOSV22,GP05,Haa13}). These estimates give us upper bounds for the  error of the differences $\yrm^{k}(\bmu)-\yrm^{k,\ell}(\bmu)$ and $\qrm^{k}(\bmu)-\qrm^{k,\ell}(\bmu)$ without actually evaluating the FE solutions (whose evaluation can lead up to long computation time). While it is possible to use residual-based norms (as in the case of the error estimate in \cite{GP05}), in the non-linear framework this is not online efficient, since the evaluation of the coupling term needs to be done in the full FE dimensions $n$.

In this section we will utilize hierarchical error estimators which are well-known, e.g., for adaptive finite elements; cf. \cite{BS93,CO96,DSM12}, for instance. The idea is to use the difference between two approximations with different orders to estimate the RB error. Here we define (approximated) hierarchical error estimators for the RB method applying the ideas in \cite{HORU19}. Given RB spaces $\Vl$ and $\Vcl$, with respective reduced bases associated with the matrices $\Psiyl$ and $\Psiql$, we define the RB errors as
\begin{equation*}
    E_\yrm^\ell(\bmu):=\left(\sum_{k=1}^K\alpha_k\,{\|\yrm^{k}(\bmu)-\yrm^{k,\ell}(\bmu)\|}^2_{\Syrm}\right)^{1/2}
    \quad\text{and}\quad
    E_\qrm^\ell(\bmu):=\left(\sum_{k=1}^K\alpha_k\,{\|\qrm^{k}(\bmu)-\qrm^{k,\ell}(\bmu)\|}^2_{\Sqrm}\right)^{1/2},
\end{equation*}
where the $\alpha_k$'s are trapezoidal weights, $\Syrm=((\langle\varphi_j,\varphi_i\rangle_V))_{0\le i,j\le n}$ and $\Sqrm=((\langle\varphi_j,\varphi_i\rangle_\Vc))_{1\le i,j\le n}$ are positive definite, symmetric weighting matrices. These errors approximate the errors $y^h-y^\ell$ and $q^h-q^\ell$ in the $L^2(0,T;V)$- and $L^2(0,T;\Vc)$-norms, respectively. If we had another couple of RB spaces, $\Vm$ and $\Vcm$ with $\elly<m_\yrm\ll n$ and $\ellq<m_\qrm\ll n$ such that $E^m_\yrm(\bmu)\le\varepsilon$ and $E^m_\qrm(\bmu)\le\varepsilon$, then using the triangular inequality we get
\begin{align*}
    E^\ell_\yrm(\bmu)^2 &= \sum_{k=1}^K \alpha_k\,{\|\yrm^{k}(\bmu)-\yrm^{k,\ell}(\bmu)\|}^2_{\Syrm}\le\sum_{k=1}^K\alpha_k\,{\|\yrm^{k}(\bmu)-\yrm^{k,m}(\bmu)\|}^2_{\Syrm}+\sum_{k=1}^K\alpha_k\,{\|\yrm^{k,m}(\bmu)-\yrm^{k,\ell}(\bmu)\|}^2_{\Syrm}\\
    &\le\varepsilon^2+\Delta_{\yrm}^{\ell,m}(\bmu)^2
\end{align*}
and, similarly, $E^\ell_\qrm(\bmu)^2\le\varepsilon^2+\Delta_\qrm^{\ell,m}(\bmu)^2$ with the computable quantities
\begin{equation*}
    \Delta_{\yrm}^{\ell,m}(\bmu)^2:=\sum_{k=1}^K\alpha_k\,{\|\yrm^{k,m}(\bmu)-\yrm^{k,\ell}(\bmu)\|}^2_{\Syrm}
    \quad\text{and}\quad
    \Delta_{\qrm}^{\ell,m}(\bmu)^2:=\sum_{k=1}^K\alpha_k\,{\|\qrm^{k,m}(\bmu)-\qrm^{k,\ell}(\bmu)\|}^2_{\Sqrm}.
\end{equation*}
These estimates are online-efficient (i.e., their evaluation does not depend on the FE dimension $n$), since
\begin{align*}
    \Delta_{\yrm}^{\ell,m}(\bmu)^2&=\sum_{k=1}^K\alpha_k\left(\hat\yrm^{k,m}(\bmu)^\top\Srm_\yrm^{m,m}\hat\yrm^{k,m}(\bmu)-2\hat\yrm^{k,m}(\bmu)^\top\Srm_\yrm^{m,\ell}\hat\yrm^{k,\ell}(\bmu)+\yrm^{k,\ell}(\bmu)^\top\Srm_\yrm^{\ell,\ell}\hat\yrm^{k,\ell}(\bmu)\right),\\
    \Delta_{\qrm}^{\ell,m}(\bmu)^2&=\sum_{k=1}^K\alpha_k\left(\hat\qrm^{k,m}(\bmu)^\top\Srm_\qrm^{m,m}\hat\qrm^{k,m}(\bmu)-2\hat\qrm^{k,m \top}(\bmu)\Srm^{m,\ell}_\qrm\hat\qrm^{k,\ell}(\bmu)+\qrm^{k,\ell}(\bmu)^\top\Srm_\qrm^{\ell,\ell}\hat\qrm^{k,\ell}(\bmu)\right),
\end{align*}
where $\{\hat\yrm^{k,\ell}(\bmu),\hat\qrm^{k,\ell}(\bmu)\}_{k=1}^K$ (resp. $\{\hat\yrm^{k,m}(\bmu),\hat\qrm^{k,m}(\bmu)\}_{k=1}^K$) are the solution of \eqref{3.2:ROM} in dimensions $\elly$ and $\ellq$ (resp. $m_\yrm$ and $m_\qrm$), $\Srm_\yrm^{m,m}=\Psi_\yrm^{m\top}\Syrm\Psiym\in\mathbb{R}^{m_\yrm\times m_\yrm}$, $\Srm_\yrm^{m,\ell}=\Psi_\yrm^{m\top}\Syrm\Psiyl\in\mathbb{R}^{m_\yrm\times \elly}$, $\Srm_\yrm^{\ell,\ell}=\Psi_\yrm^{\ell\top}\Syrm\Psiyl\in\mathbb{R}^{\elly\times\elly}$, and the matrices $\Srm_\qrm^{m,m}$, $\Srm_\qrm^{m,\ell}$, $\Srm_\qrm^{\ell,\ell}$ are defined similarly.

To analyze the hierarchical error estimates it is useful to set $\Vl\subset\Vm$ and $\Vcl\subset\Vcm$ and to assume the following \emph{saturation property} on the bigger RB space $\Vm$: there exist $\sigma_\yrm,\sigma_\qrm \in (0,1)$ such that
\begin{equation}
    \label{2.5:saturation_assumption}
    E^m_\yrm(\bmu)^2 \le \sigma_\yrm E^\ell_\yrm(\bmu)^2
    \quad\text{and}\quad
    E^m_\qrm(\bmu)^2 \le \sigma_\qrm E^\ell_\qrm(\bmu)^2
    \qquad\text{for all }\bmu \in \Pad.
\end{equation}
This is a natural assumption since one expects the RB error to decrease with the increase of the RB dimension, but this condition will be enforced during the construction of the spaces $\Vm$ and $\Vcm$.
\begin{prop}
    If \eqref{2.5:saturation_assumption} holds, then for all $\bmu \in \Pad$
    \begin{subequations}
        \label{PropEst}
        \begin{align}
            \label{PropEst-a}
            \frac{\Delta_\yrm^{\ell,m}(\bmu)^2}{1+\sigma_\yrm}&\le E_\yrm^\ell(\bmu)^2=\sum_{k=1}^K\alpha_k\,{\|\yrm^{k}(\bmu)-\yrm^{k,\ell}(\bmu)\|}^2_{\Syrm}\le\frac{\Delta_\yrm^{\ell,m}(\bmu)^2}{1-\sigma_\yrm},\\
            \label{PropEst-b}
            \frac{\Delta_\qrm^{\ell,m}(\bmu)^2}{1+\sigma_\qrm}&\le E_\qrm^\ell(\bmu)^2=\sum_{k=1}^K\alpha_k\,{\|\qrm^{k}(\bmu)-\qrm^{k,\ell}(\bmu)\|}^2_{\Sqrm}\le\frac{\Delta_\qrm^{\ell,m}(\bmu)^2}{1-\sigma_\qrm}.
        \end{align}
    \end{subequations}
\end{prop}

\begin{proof}
    We only prove \eqref{PropEst-a}. The proof of \eqref{PropEst-b} is similar. If $\|\yrm^{k}(\bmu)-\yrm^{k,\ell}(\bmu)\|_{\Syrm}=0$ for $k=1,\ldots,K$, then $\|\yrm^{k}(\bmu)-\yrm^{k,m}(\bmu)\|_{\Syrm}=0$  for $k=1,\ldots,K$ as well. Thus we get $\Delta_{\yrm}^{\ell,m}(\bmu)=0$ so that \eqref{PropEst-a} follows. Otherwise, using the reverse triangle inequality and \eqref{2.5:saturation_assumption} we get
    \begin{align*}
        \frac{\sum_{k=1}^K\alpha_k\,{\|\yrm^{k,\ell}(\bmu)-\yrm^{k,m}(\bmu)\|}^2_{\Syrm}}{\sum_{k=1}^K\alpha_k\,{\|\yrm^{k}(\bmu)-\yrm^{k,\ell}(\bmu)\|}^2_{\Syrm}}
        &\ge\frac{\sum_{k=1}^K\alpha_k\,{\|\yrm^{k}(\bmu)-\yrm^{k,\ell}(\bmu)\|}^2_{\Syrm}-\sum_{k=1}^K\alpha_k\,{\|\yrm^{k}(\bmu)-\yrm^{k,m}(\bmu)\|}^2_{\Syrm}}{\sum_{k=1}^K\alpha_k\,{\|\yrm^{k}(\bmu)- \yrm^{k,\ell}(\bmu)\|}^2_{\Syrm}} \\
        &\ge 1-\frac{\sum_{k=1}^K\alpha_k\,{\|\yrm^{k}(\bmu)-\yrm^{k,m}(\bmu)\|}^2_{\Syrm}}{\sum_{k=1}^K\alpha_k\,{\|\yrm^{k}(\bmu)-\yrm^{k,\ell}(\bmu)\|}^2_{\Syrm}}=1-\frac{E^m_\yrm(\bmu)^2}{E^\ell_\yrm(\bmu)^2}\ge 1-\sigma_\yrm.
    \end{align*}
    On the other hand, using the triangle inequality and \eqref{2.5:saturation_assumption},
    \begin{align*}
        \frac{\sum_{k=1}^K\alpha_k\,{\|\yrm^{k,\ell}(\bmu)-\yrm^{k,m}(\bmu)\|}^2_{\Syrm}}{\sum_{k=1}^K\alpha_k\,{\|\yrm^{k}(\bmu)-\yrm^{k,\ell}(\bmu)\|}^2_{\Syrm}}
        &\le\frac{\sum_{k=1}^K\alpha_k\,{\|\yrm^{k}(\bmu)-\yrm^{k,\ell}(\bmu)\|}^2_{\Syrm}+\sum_{k=1}^K\alpha_k\,{\|\yrm^{k}(\bmu)-\yrm^{k,m}(\bmu)\|}^2_{\Syrm}}{\sum_{k=1}^K\alpha_k\,{\|\yrm^{k}(\bmu)-\yrm^{k,\ell}(\bmu)\|}^2_{\Syrm}} \\
        &\le 1 +\frac{\sum_{k=1}^K\alpha_k\,{\|\yrm^{k}(\bmu)-\yrm^{k,m}(\bmu)\|}^2_{\Syrm}}{\sum_{k=1}^K\alpha_k\,{\|\yrm^{k}(\bmu)-\yrm^{k,\ell}(\bmu)\|}^2_{\Syrm}}=1+\frac{E^m_\yrm(\bmu)^2}{E^\ell_\yrm(\bmu)^2}\le 1+\sigma_\yrm,
    \end{align*}
    which gives \eqref{PropEst-a}.
\end{proof}

Next, we define the hierarchical error estimators in $\bmu\in\Pad$ as
\begin{equation}
    \label{EstimateEff}
    \Delta_\yrm^\ell(\bmu)=\Delta_\yrm^\ell(\bmu;m,\sigma_\yrm) = \frac{\Delta_\yrm^{\ell,m}(\bmu)}{\sqrt{1-\sigma_\yrm}}
    \quad\text{and}\quad
    \Delta_\qrm^\ell(\bmu)=\Delta_\qrm^\ell(\bmu;m,\sigma_\qrm) = \frac{\Delta_\qrm^{\ell,m}(\bmu)}{\sqrt{1-\sigma_\qrm}}.
\end{equation}

\begin{rem}
    The effectivities of the estimate in the parameter $\bmu\in\Pad$ are defined as
    \begin{equation}
        \label{3.3:max_efficiency}
        \eta_\yrm(\bmu):=\frac{\Delta_\yrm^\ell(\bmu)}{E_\yrm^\ell(\bmu)}
        \quad\text{and}\quad
        \eta_\qrm(\bmu):=\frac{\Delta_\qrm^\ell(\bmu)}{E_\qrm^\ell(\bmu)}.
    \end{equation}
    Then, we infer from \eqref{3.3:max_efficiency}, \eqref{PropEst} and \eqref{EstimateEff} that
    \begin{equation}
        \label{3.3:approx_max_efficiency}
        1\le\eta_\yrm(\bmu)\le\bar\eta_\yrm:=\sqrt{\frac{1+\sigma_\yrm}{1-\sigma_\yrm}}
        \quad\text{and}\quad
        1\le\eta_\qrm(\bmu)\le\bar\eta_\qrm:=\sqrt{\frac{1+\sigma_\qrm}{1-\sigma_\qrm}}
        \qquad\text{for all }\bmu\in\Pad.
    \end{equation}
    Hence we observe that the closer $\sigma_\yrm$ (respectively $\sigma_\qrm$) is to $0$, the more accurate the estimate $\Delta^\ell_\yrm(\bmu)$  (resp. $\Delta^\ell_\qrm(\bmu)$) is for any $\bmu\in\Pad$.\hfill$\Diamond$
\end{rem}

There are two important choices to be made: the first is how to compute $\sigma_\yrm$ and $\sigma_\qrm$, and the second is how to choose the RB spaces $\Vm$ and $\Vcm$. For the first issue, we compute approximations $\tilde\sigma_\yrm$ and $\tilde\sigma_\qrm$ on a training set, namely
\begin{equation}
    \label{3.3:evaluation_sigmas}
    \sigma_\yrm\approx\tilde\sigma_\yrm:=\max_{\bmu\in\Ptr}\frac{E^m_\yrm(\bmu)^2}{E^\ell_\yrm(\bmu)^2}
    \quad\text{and}\quad
    \sigma_\qrm\approx\tilde\sigma_\qrm:=\max_{\bmu\in\Ptr} \frac{E^m_\qrm(\bmu)^2}{E^\ell_\qrm(\bmu)^2}.
\end{equation}
Let us observe that the computation of \eqref{3.3:evaluation_sigmas} can be done efficiently by parallelization.

Next, the second question is answered by defining $\Vm$ (resp. $\Vcm$) as an expansion of $\Vl$ (resp. $\Vcl$), namely $\Vl\subset\Vm$ and $\Vcl\subset\Vcm$. In particular, given $\Vl$ and $\Vcl$, we will find $\tilde{V}^m\perp\Vl$ and $\tilde{V}^m_\circ\perp\Vcl$ and then define
\begin{equation}
    \label{3.3:expansion}
    \Vm=\Vl\oplus\tilde V^m
    \qquad\text{and}\qquad
    \Vcm=\Vcl\oplus\tilde V^m_\circ,
\end{equation}
where $\oplus$ is the direct sum (cf. e.g., \cite{BS93}). The greedy algorithm for the evaluation of the RB spaces is shown in Algorithm \ref{3.3:algorithm_RBHE}.
The input of the algorithm are the tolerance $\tau>0$, an initial parameter $\bmu=(\mu_i)_{1\le i\le 4}\in\Pad$, a training set $\Ptr$ and the maximum RB cardinality $L>0$. The maximum number of bases elements is a safeguard in case the maximum error does not go under tolerance.

\begin{algorithm}
    \caption{(Weak Greedy algorithm)}
    \label{3.3:algorithm_RBHE}
	\begin{algorithmic}[1] 
        \REQUIRE $\tau>0$, $\hat\bmu\in\Pad$, $\Ptr\subset\Pad$, $L>0$;
        \STATE Create RB spaces $\Vl$, $\Vcl$, $\tilde{V}^m$ and $\tilde{V}^m_\circ$ in $\hat\bmu$;\label{3.3alg:line_1}
        \STATE $\Vm\gets\Vl\oplus\tilde V^m$ and $\Vcm\gets\Vcl\oplus \tilde V^m_\circ$ as in \eqref{3.3:expansion};\label{3.3alg:line_2}
        \STATE Evaluate $\tilde\sigma_\yrm$ and $\tilde\sigma_\qrm$ as in \eqref{3.3:evaluation_sigmas};\label{3.3alg:line_3}
        \WHILE{$\tilde\sigma_\yrm\ge1$ \textbf{or} $\tilde\sigma_\qrm\ge1$}\label{3.3alg:line_4}
            \STATE Enrich $\tilde{V}^m$ and/or $\tilde{V}^m_\circ$;%\label{3.3alg:line_5}
            \STATE Set $\Vm \gets \Vl \oplus \tilde{V}^m$ and $\Vcm \gets \Vcl \oplus \tilde{V}^m_\circ$;\label{3.3alg:line_6}
            \STATE Evaluate $\tilde\sigma_\yrm$ and $\tilde\sigma_\qrm$ as in \eqref{3.3:evaluation_sigmas};\label{3.3alg:line_7}
       \ENDWHILE
        \STATE Evaluate $\hat{\bmu}=\argmax_{\bmu\in\Ptr}e^\ell(\bmu):=(\Delta^\ell_\yrm(\bmu)+\Delta^\ell_\qrm(\bmu))/2$ and $\hat{e}=e(\hat\bmu)$;\label{3.3alg:line_9}
        \WHILE{$\hat{e}>\tau$ \textbf{and} $\elly+\ellq<L$}%\label{3.3alg:line_10}
            \IF{$\Delta^\ell_\yrm(\hat\bmu)>\tau$}%\label{3.3alg:line_11}
                \STATE Enrich $\Vl$ in $\hat\bmu$;%\label{3.3alg:line_12}
            \ENDIF
            \IF{$\Delta^\ell_\qrm(\hat\bmu)>\tau$}\label{3.3alg:line_13}
                \STATE Enrich $\Vcl$ in $\hat\bmu$;%\label{3.3alg:line_15}
            \ENDIF
            \STATE Set $\Vm \gets \Vl \oplus \tilde{V}^m$ and $\Vcm \gets \Vcl \oplus \tilde{V}^m_\circ$;%\label{3.3alg:line_17}
            \WHILE{$\tilde\sigma_\yrm\ge1$ \textbf{or} $\tilde\sigma_\qrm\ge1$}\label{3.3alg:line_18}
                \STATE Enrich $\tilde{V}^m$ and $\tilde{V}^m_\circ$;%\label{3.3alg:line_19}
                \STATE Set $\Vm \gets \Vl \oplus \tilde{V}^m$ and $\Vcm \gets \Vcl \oplus \tilde{V}^m_\circ$;%\label{3.3alg:line_20}
                \STATE Evaluate $\tilde\sigma_\yrm$ and $\tilde\sigma_\qrm$ as in \eqref{3.3:evaluation_sigmas};\label{3.3alg:line_21}
            \ENDWHILE
            \STATE Evaluate $\hat{\bmu}=\argmax_{\bmu\in\Ptr}e^\ell(\bmu):=(\Delta^\ell_\yrm(\bmu)+\Delta^\ell_\qrm(\bmu))/2$ and $\hat{e}=e(\hat\bmu)$;\label{3.3alg:line_23}
        \ENDWHILE
     \end{algorithmic}
\end{algorithm}

In line \ref{3.3alg:line_2} the computation of the initial RB spaces is described. As a first step, we create RB spaces $\Vl$ and $\Vcl$ using POD with the state variables' snapshots $\{\yrm^k(\bmu)\}_{k=1}^K$ and $\{\qrm^k(\bmu)\}_{k=1}^K$. Then, to define $\tilde{V}^m$ and $\tilde{V}^m_\circ$ we use the sensitivity variables (as suggested in \cite{HORU19}). In particular, we first evaluate the vectors $\{\mathrm{s}^k_{\yrm,i}(\bmu),\mathrm{s}^k_{\qrm,i}(\bmu)\}_{k=1}^K$ for $i=1,\dots,4$, where $\mathrm{s}^k_{\yrm,i}(\bmu)$ is the coordinate vector of FE approximation of the $i$-th sensitivity variable w.r.t. state variable $y$ evaluated at $t_k$ using parameter $\bmu$. Secondly, we orthogonalize the vectors $\{\yrm^k(\bmu),\,\mathrm{s}^k_{\yrm,i}(\bmu)\}_{k=1}^K$ (resp. $\{\qrm^k(\bmu),\,\mathrm{s}^k_{\qrm,i}(\bmu)\}_{k=1}^K$) for $i=1,\dots,4$ w.r.t. $\Vl$ (resp. $\Vcl$). Finally, we use POD on the two sets of vectors to evaluate $\tilde{V}^m$ and $\tilde{V}^m_\circ$. Then, line \ref{3.3alg:line_6} uses \eqref{3.3:expansion} to define the spaces $\Vm$ and $\Vcm$. We observe that by construction $\Vl\perp\tilde{V}^m$ and $\Vcl\perp\tilde{V}^m_\circ$, though it is possible to orthogonalize the bases to enforce this condition numerically. Furthermore, we add that at first we choose $m_\yrm=\elly+2$ and $m_\qrm=\ellq+2$, namely we construct the initial $\tilde{V}^m$ and $\tilde{V}^m_\circ$ of dimension 2.

Lines \ref{3.3alg:line_7} and \ref{3.3alg:line_21} evaluates $\tilde\sigma_\yrm$ and $\tilde\sigma_\qrm$. If $\tilde\sigma_\yrm>1$ or $\tilde\sigma_\qrm>1$ (as in line \ref{3.3alg:line_4}), it means that the saturation property is violated. This has so far not happened to us, confirming the expectation that a bigger RB space is a better approximation of the FE space for any $\bmu\in\Pad$. But in case this condition is violated, our algorithm enriches the spaces $\tilde{V}^m$ and $\tilde{V}^m_\circ$ in the parameters where such condition was violated. Enriching here means evaluating FE vectors, orthogonalizing w.r.t. $\Vm$ or $\Vcm$, and then using POD to add some new bases to them. This procedure is then applied as long as the saturation property is violated.

Lines \ref{3.3alg:line_9} and \ref{3.3alg:line_23} evaluate the error measure we want to lower. If this error is higher than a certain tolerance, it means that either $\Delta^\ell_\yrm(\hat\bmu)>\tau$ or $\Delta^\ell_\qrm(\hat\bmu)>\tau$ (or, eventually, both). Then the enrichment is done by evaluating the states' snapshots, orthogonalizing them w.r.t. the bases we already have and then using POD to get new elements. We can then add as many elements such that $\Delta^\ell_\yrm(\hat\bmu)\le\tau$ and $\Delta^\ell_\qrm(\hat\bmu)\le\tau$.

We then check $\tilde\sigma_\yrm$ and $\tilde\sigma_\qrm$ again. Since $\Vl$ and $\Vcl$ could have been enriched, we need to orthogonalize $\tilde{V}^m$ and $\tilde{V}^m_\circ$ with respect to them. It still holds that $\Vl\subset\Vm$ and $\Vcl\subset\Vcm$, but if the saturation property fails, we enforce it by enriching $\tilde{V}^m$ or $\tilde{V}^m_\circ$.

%%-----------------------------
\subsection{Numerical experiments}
\label{Sec:4.2}
%%-----------------------------

Let us now set $\Omega=(0,1)$ and $T=1$, discretized with linear Lagrangian elements in space on 200 spatial nodes (so that $n=201$) and the implicit Euler method for $K=201$ time steps. We choose the diffusion functions $\kappa_1(x)\equiv\kappa_2(x)\equiv1$, initial value $y_\circ(x)\equiv5$ and parameter bounds as $\muai=1$ and $\mubi=5$ for $i=1,\dots,4$. The set $\Ptr$ is chosen as a uniform grid $5\times5\times5\times5$ on $\Pad$, such that $\Ptr=\{\bmu\in\mathbb{R}^4\,|\,\mu_i\in\{1,2,3,4,5\}\text{ for }i=1,2,3,4\}$. Furthermore, we choose $\tau=10^{-4}$ and $L=50$.

\begin{figure}[t]
\centering
\begin{tabular}{rr}
    \includegraphics{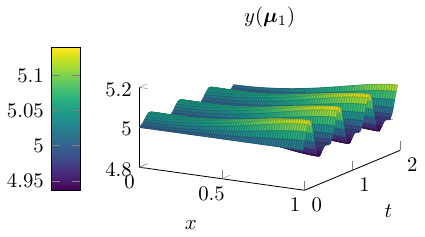}
    &
    \includegraphics{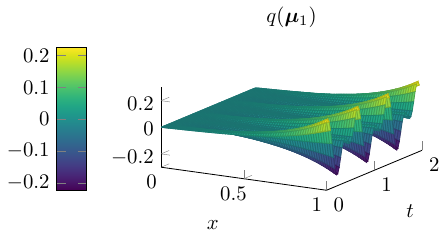}
    \\
    \includegraphics{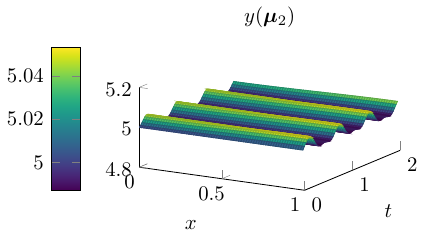}
    &
    \includegraphics{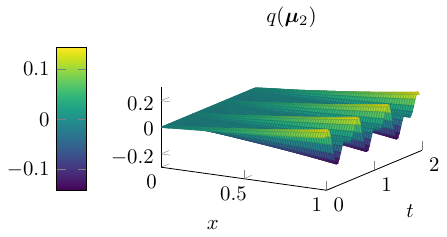}
\end{tabular}
\caption{\label{Fig:plots_different_parameters} State variables for $\bmu_1=(1,5,1,5)$ and $\bmu_2=(5,3,4,2)$.}
\end{figure}

We analyse the results of the greedy algorithm and the generated RB spaces using three different input functions, namely
\begin{align*}
    u_1(t)\equiv1,\quad u_2(t)=-\bm{1}_{[0,0.75)}(t)+\bm{1}_{[0.75,1]}(t),\quad u_3(t)=0.5\cdot\cos(10t)+0.4\cdot\sin(20t)\quad\text{for }t\in[0,T],
\end{align*}
where $\bm{1}_I(t)$ is the indicator function of interval $I$.
To give an idea of what state variables look like, we show in Figure~\ref{Fig:plots_different_parameters} approximated FE states for input function $u_3$.

We will compare the overall time of the Algorithm~\ref{3.3:algorithm_RBHE}, the number of bases generated in all spaces, the approximated values $\tilde\sigma_\yrm$ and $\tilde\sigma_\qrm$ and the respective approximated efficiency, $\tilde\eta_\yrm$ and $\tilde\eta_\qrm$ given as \eqref{3.3:approx_max_efficiency}; cf. Table~\ref{Tab:greedy_results}. Furthermore, given a set $\Pte$ of 100 random parameters in $\Pad$, we evaluate the average times of evaluating FE and RB states, maximum test errors
\begin{equation*}
    \hat{E}^\ell_\yrm = \max_{\bmu\in\Pte} E^\ell_\yrm(\bmu)
    \qquad\text{and}\qquad
    \hat{E}^\ell_\qrm = \max_{\bmu\in\Pte} E^\ell_\qrm(\bmu)
\end{equation*}
and maximum test efficiencies
\begin{equation*}
    \hat\eta_\yrm = \max_{\bmu\in\Pte} \eta_\yrm(\bmu)
    \qquad\text{and}\qquad
    \hat\eta_\qrm = \max_{\bmu\in\Pte} \eta_\qrm(\bmu),
\end{equation*}
where $\eta_\yrm(\bmu)$ and $\eta_\qrm(\bmu)$ are defined in \eqref{3.3:max_efficiency}.

\begin{table}[ht]
\begin{center}
{\small\begin{tabular}{|l|c|c|c|} 
\hline
& $u_1$ & $u_2$ & $u_3$ \\\hline
Time greedy & 3424 $s$ & 3456 $s$ & 3704 $s$ \\
Number of bases of $\Vl,\Vcl$: $\ell_\yrm$, $\ell_\qrm$ & 8, 4 & 8, 4 & 7, 4 \\
Number of bases of $\Vm,\Vcm$: $m_\yrm$, $m_\qrm$ & 10, 6 & 10, 6 & 9, 6 \\
Approximated $\sigma$'s: $\tilde\sigma_\yrm$, $\tilde\sigma_\qrm$ & 0.18, 0.46 & 0.22, 0.05 & 0.03, 0.67 \\
Approx. max. efficiency $\tilde\eta_\yrm$, $\tilde\eta_\qrm$ & 1.20, 1.64 & 1.25, 1.05 & 1.03, 2.26 \\
Avg. time FE & 1.50 $s$ & 1.56 $s$ & 1.58 $s$ \\
Avg. time RB & 0.05 $s$ & 0.05 $s$ & 0.08 $s$ \\
Maximum test errors: $\hat{E}^\ell_\yrm$, $\hat{E}^\ell_\qrm$ & 7.22e-06, 1.19e-05 & 6.31e-06, 2.28e-05 & 7.38e-06, 2.33e-05 \\
Maximum test efficiencies: $\hat\eta_\yrm$, $\hat\eta_\qrm$ & 1.12, 1.36 & 1.18, 1.03 & 1.01, 1.74 \\
\hline
\end{tabular}}
\end{center}
\caption{Comparison of the test runs for three different control inputs $u$.\label{Tab:greedy_results}}
\end{table}

\begin{figure}[ht]
\centering
\begin{tabular}{rrr}
    \includegraphics{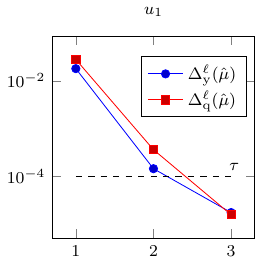}
    &
    \includegraphics{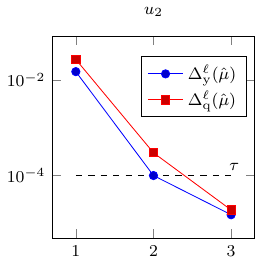}
    &
    \includegraphics{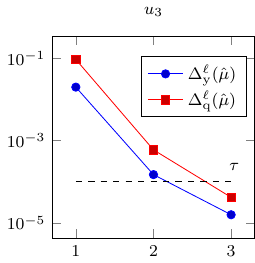}
    \\
    \includegraphics{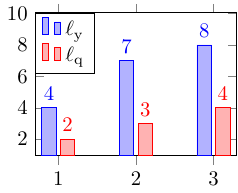}
    &
    \includegraphics{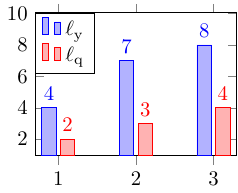}
    &
    \includegraphics{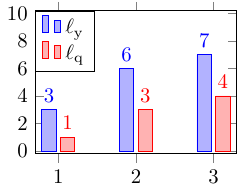}
\end{tabular}
\caption{\label{Fig:greedy}Visual comparison of the Greedy results for the three different control inputs. In both rows, the current greedy iteration is indicated on the $x$-axis: in the first row we can see how the convergence of the RB error estimators under tolerance is reached in three iterations, while in the second row we see the growing of the dimensions of the RB spaces.}
\end{figure}

The overall algorithm time does not change much with different inputs.
The evaluation of $\hat{e}$ in lines \ref{3.3alg:line_9} and \ref{3.3alg:line_23} does not take much CPU time. On the contrary, the evaluation of \eqref{3.3:evaluation_sigmas} takes necessarily more time, because the evaluation of $\tilde\sigma_\yrm$ and $\tilde\sigma_\qrm$ needs the evaluation of the FE solutions, while the evaluation of $e^\ell(\bmu)$ does not. We can see that the bigger spaces $\Vm$ and $\Vcm$ are in fact of dimensions $m_\yrm=\elly+2$ and $m_\qrm=\ellq+2$, meaning that once $\tilde{V}^m$ and $\tilde{V}^m_\circ$ are created at the beginning they are not enriched anymore, namely no more bases are added to them since the conditions on lines \ref{3.3alg:line_4} and \ref{3.3alg:line_18} of Algorithm~\ref{3.3:algorithm_RBHE} are never true. The value of $\tilde\sigma_\yrm$ and $\tilde\sigma_\qrm$ change significantly in the different tests, but in the end we see that they hold good approximated max efficiencies $\tilde\eta_\yrm$ and $\tilde\eta_\qrm$, which we can read as ``efficiency estimators'' for the test efficiencies $\hat\eta_\yrm$ and $\hat\eta_\qrm$.

We conclude with a comment on the hierarchical estimators. Even if, as we have stressed, these are just approximations of the estimators, they seem tight even with a relatively small effort: indeed, if the time of evaluation of an RB solution amounts to $5\%$ of the time evaluating a FE solution, we can expect the time to evaluate the estimator to amount to around $10\%$ of the time evaluating the true error (since the error estimator consists in evaluating two RB solutions).

%%-----------------------------
\section{The parameter optimization}
\label{Sec:5}
%------------------------------

Now we are interested to solve numerically the non-linear PDE-constrained parameter optimization problem
\begin{equation*}
    \tag{$\mathbf P$}\label{5.0:optimization_problem}
    \min J(y,q,\bmu) \quad \text{s.t.}\quad (y,q)\in\Yscr\times\Qscr \text{ is a weak solution of \eqref{2.0:coupled_system_together} and } \bmu\in\Pad.
\end{equation*}

Problem \eqref{2.1:weak_all} is uniquely solvable only locally in time. For that reason, we make use of the following hypothesis.

\begin{assum}
    \label{Assum:4}
    For given final time $T>0$ there exists a unique solution pair $(y,q)\in\Yscr\times\Qscr$ to \eqref{2.1:weak_all} for any $\bmu\in\Pad$ which is denoted as $(y(\bmu),q(\bmu))$.
\end{assum}

For $\Wscr:=L^2(0,T;H)$ the cost functional is defined as the $\Wscr$-norm of the error between the observable function $\eta:\Yscr\times\Qscr\times\Pad\to\Wscr$ and a given data function $w\in\Wscr$ that depends on $\bmu^*\in\Pad$, the underlying parameter, plus a regularization term. Then,
\begin{equation}
    J(y,q,\bmu) = \frac{\alpha^{J}}{2}\,{\|\eta(y,q,\bmu)-w\|}^2_\Wscr+\frac{\lambda}{2}\,{\|\bmu-\hat{\bmu}\|}_2^2\quad\text{for }\bmu\in\Pad,
\end{equation}
where $\alpha^J>0$, $\lambda>0$ hold and $\hat\bmu\in\Pad$ is a reference parameter. Then, setting $\hat\eta(\bmu):=\eta(y(\bmu),q(\bmu),\bmu)$, we define the reduced cost function as

\begin{equation}
    \label{4.0:reduced_cost_function}
    \hat{J}(\bmu)=\frac{\alpha^{J}}{2}\,{\|\hat\eta(\bmu)-w\|}^2_\Wscr+\frac{\lambda}{2}\,{\|\bmu-\hat{\bmu}\|}_2^2.
\end{equation}

In many applications including battery systems, we only have limited information and can only measure one of the two states. We assume our case to be similar, namely we can only measure the state $q(\bmu)$, so that $\hat\eta(\bmu)=q(\bmu)$. Finally, we rewrite the optimization problem in reduced form as
\begin{equation}
    \tag{$\hat{\mathbf{P}}$}\label{5.0:reduced_optimization_problem}
    \min\hat{J}(\bmu)\quad\text{s.t.}\quad\bmu\in\Pad.
\end{equation}

The FE and the RB approximations will be, respectively,

\begin{equation*}
    \hat{J}^h(\bmu)=\frac{\alpha^{J}}{2}\sum_{k=1}^K\alpha_k\,{\|\qrm^k(\bmu)-\wrm^k\|}_{\mathrm M_\qrm}^2+\frac{\lambda}{2}\,{\|\bmu-\hat{\bmu}\|}_2^2
\end{equation*}
and
\begin{equation*}
    \hat{J}^\ell(\bmu)=\frac{\alpha^{J}}{2}\sum_{k=1}^K\alpha_k\,{\|\qrm^{k,\ell}(\bmu)-\wrm^k\|}_{\mathrm M_\qrm}^2+\frac{\lambda}{2}\,{\|\bmu-\hat{\bmu}\|}_2^2,
\end{equation*}
where $\alpha_0,\dots,\alpha_K$ are the trapezoidal weights for the discretization of the temporal integral and $\wrm^k$ is the coordinate vector of the $H$-projection of
\begin{align*}
w^k=\frac{1}{\Delta t}\int_{t_k-\Delta t/2}^{t_k+\Delta t/2}w(s)\,\mathrm ds\quad\text{for }k=1,\ldots,K\text{ and }t_k=k\Delta t
\end{align*}
onto $\Vc$.

\begin{rem}
    The evaluation of $\hat{J}^\ell$ is online efficient, since pre-evaluating
    \begin{equation*}
        \mathrm{r}_1^k=\Psiq^\top\mathrm{M}_\qrm\wrm^k\in\mathbb{R}^\ellq
        \qquad\text{and}\qquad
        \mathrm{r}_2^k=\wrm^{k\top}\mathrm{M}_\qrm\wrm^k\in\mathbb{R}
        \qquad \text{for } k=1,\dots,K,
    \end{equation*}
    then
    \begin{equation*}
        {\|\qrm^{k,\ell}(\bmu)-\wrm^k\|}_{\mathrm{M}_\qrm}^2=(\Psi_\qrm\hat\qrm^k(\bmu)-\wrm^k)^\top\mathrm{M}_\qrm(\Psi_\qrm\hat\qrm^k(\bmu)-\wrm^k)=\hat\qrm^k(\bmu)^\top\mathrm{M}_\qrm^\ell \hat\qrm^k(\bmu) - 2 \hat\qrm^k(\bmu)^\top\mathrm{r}_1^k+\mathrm{r}_2^k.
    \end{equation*}
\end{rem}

Next, we show an approximated a-posteriori error estimator for the reduced cost using the hierarchical error estimator $\Delta^\ell_\qrm(\bmu)$.

\begin{prop}
    For all $\bmu\in\Pad$, the following estimate on the RB error of the cost functional holds:
    \begin{equation*}
        \left|\hat{J}^h(\bmu)-\hat{J}^\ell(\bmu)\right|\le\Delta^\ell_J(\bmu),
    \end{equation*}
    where
    \begin{equation}
        \label{5.0:Delta_J}
        \Delta^\ell_J(\bmu)=\frac{\alpha^JL^4}{2\pi^4}\Delta^\ell_\qrm(\bmu)^2+\frac{\alpha^JL^2}{\pi^2}\Delta^\ell_\qrm(\bmu)\sqrt{\tilde{J}^\ell(\bmu)}
    \end{equation}
and $\tilde{J}^\ell(\bmu)=\sum_{k=1}^K\alpha_k\,\|\qrm^{k,\ell}(\bmu)-\wrm^k\|_{\mathrm{M}_\qrm}^2$.
\end{prop}

\begin{proof}
As a first step, the following equality holds
\begin{equation}
    \label{5.0:proof_DeltaJ_1}
    \hat{J}^h(\bmu)-\hat{J}^\ell(\bmu)=\frac{\alpha^{J}}{2}\sum_{k=1}^K\alpha_k\,\left(\langle \qrm^k(\bmu)-\wrm^k,\qrm^k(\bmu)-\wrm^k\rangle_{\mathrm{M}_\qrm}-\langle\qrm^{k,\ell}(\bmu)-\wrm^k,\qrm^{k,\ell}(\bmu)-\wrm^k\rangle_{\mathrm{M}_\qrm}\right).
\end{equation}
Then, adding and subtracting the term $\langle\qrm^k(\bmu)-\wrm^k,\qrm^{k,\ell}(\bmu)-\wrm^k\rangle_{\mathrm{M}_\qrm}$ to each element in the sum in \eqref{5.0:proof_DeltaJ_1} we can write
\begin{align*}
    &\langle\qrm^k(\bmu)-\wrm^k,\qrm^k(\bmu)-\wrm^k\rangle_{\mathrm{M}_\qrm}-\langle\qrm^{k,\ell}(\bmu)-\wrm^k,\qrm^{k,\ell}(\bmu)-\wrm^k\rangle_{\mathrm{M}_\qrm}\\
    &=\langle\qrm^k(\bmu)-\wrm^k,\qrm^k(\bmu)-\qrm^{k,\ell}(\bmu)\rangle_{\mathrm{M}_\qrm}+\langle\qrm^k(\bmu)-\qrm^{k,\ell}(\bmu),\qrm^{k,\ell}(\bmu)-\wrm^k\rangle_{\mathrm{M}_\qrm}\\
    &=\langle\qrm^k(\bmu)-\qrm^{k,\ell}(\bmu),\qrm^k(\bmu)+\qrm^{k,\ell}(\bmu)-2\wrm^k\rangle_{\mathrm{M}_\qrm}\\
    &=\langle\qrm^k(\bmu)-\qrm^{k,\ell}(\bmu),\qrm^k(\bmu)-\qrm^{k,\ell}(\bmu)\rangle_{\mathrm{M}_\qrm}+2\langle\qrm^k(\bmu)-\qrm^{k,\ell}(\bmu),\qrm^{k,\ell}(\bmu)-\wrm^k\rangle_{\mathrm{M}_\qrm}.
\end{align*}
Hence,
\begin{equation}
    \label{5.0:proof_DeltaJ_2}
    |\hat{J}^h(\bmu)-\hat{J}^\ell(\bmu)|\le\frac{\alpha^J}{2}\sum_{k=1}^K\alpha_k\,\|\qrm^k(\bmu)-\qrm^{k,\ell}(\bmu)\|_{\mathrm{M}_\qrm}^2+\alpha^J\sum_{k=1}^K\alpha_k\,\left|\langle\qrm^k(\bmu)-\qrm^{k,\ell}(\bmu),\qrm^{k,\ell}(\bmu)-\wrm^k\rangle_{\mathrm{M}_\qrm}\right|.
\end{equation}
Using \eqref{2.1:poincare_inequality}, \eqref{PropEst-b} and \eqref{EstimateEff} we can estimate
\begin{equation}
    \label{5.0:proof_DeltaJ_3}
    \sum_{k=1}^K\alpha_k\,\|\qrm^k(\bmu)-\qrm^{k,\ell}(\bmu)\|_{\mathrm{M}_\qrm}^2\le c_\mathsf{P}^2 E^\ell_\qrm(\bmu)^2\le c_\mathsf{P}^2\Delta^\ell_\qrm(\bmu)^2.
\end{equation}
On the other hand, using the Cauchy-Schwarz inequality and \eqref{5.0:proof_DeltaJ_3} we get
\begin{equation}
    \begin{aligned}
        \label{5.0:proof_DeltaJ_4}
        &\sum_{k=1}^K\alpha_k\,|\langle\qrm^k(\bmu)-\qrm^{k,\ell}(\bmu),\qrm^{k,\ell}(\bmu)-\wrm^k\rangle_{\mathrm{M}_\qrm}|\\
        &\qquad\le\left(\sum_{k=1}^K\alpha_k\,\|\qrm^k(\bmu)-\qrm^{k,\ell}(\bmu)\|_{\mathrm{M}_\qrm}^2\right)^{1/2}\left(\sum_{k=1}^K\alpha_k\,\|\qrm^{k,\ell}(\bmu)-\wrm^k\|_{\mathrm{M}_\qrm}^2\right)^{1/2}\le c_\mathsf{P} \Delta^\ell_\qrm(\bmu)\sqrt{\tilde{J}^\ell(\bmu)}.
    \end{aligned}
\end{equation}
Inserting \eqref{5.0:proof_DeltaJ_3} and \eqref{5.0:proof_DeltaJ_4} into \eqref{5.0:proof_DeltaJ_2} we get
\begin{equation*}
    |\hat{J}^h(\bmu)-\hat{J}^\ell(\bmu)|\le\frac{\alpha^J c_\mathsf{P}^2}{2}\Delta^\ell_\qrm(\bmu)^2+\alpha^J c_\mathsf{P}\Delta^\ell_\qrm(\bmu)\sqrt{\tilde{J}^\ell(\bmu)}.
\end{equation*}
Using $c_\mathsf{P}=(L/\pi)^2$ the optimal Poincaré constant in $[0,L]$, we finally get \eqref{5.0:Delta_J}.
\end{proof}

%%-----------------------------
\subsection{The optimization algorithm}
\label{Sec:5.1}
%%-----------------------------

We solve the optimization problem \eqref{5.0:reduced_optimization_problem} in a trust-region (TR) framework using the RB model as a surrogate model. This was done in the recent work \cite{QGVW17,KMOSV21,BKMOSV22,PSV22}.

The TR optimization algorithm computes iteratively a first-order critical point of \eqref{5.0:reduced_optimization_problem}. At each iteration $i \ge 0$ of the optimization algorithm, a cheaply computable model $m^{(i)}$ (approximation of the reduced cost) is used to accurately represent the function $\hat{J}^h$ in a reasonable neighborhood of $\bmu^{(i)}$, called \textit{trust region} $\mathcal{T}(\delta^{(i)}) = \{ \bmu : \| \bmu - \bmu^{(i)} \|_2 \le \delta^{(i)} \} $, where $\delta^{(i)}$ is called \textit{TR radius}. The TR method finds the next iteration $\bmu^{(i+1)}$ of the optimization algorithm by solving the problem
\begin{equation}
    \label{5.1:TR_problem_general}
    \min_{s \in \mathbb R^d} m^{(i)} (s)\quad\text{s.t.}\quad {\| s \|}_2 \le \delta^{(i)}, \,\bmu^{(i)} + s \in\Pad.
\end{equation}

For $\bmu = \bmu^{(i)} + s$, the RB version of \eqref{5.1:TR_problem_general} is
\begin{equation}
    \label{eq:TR_subproblem}
    \min_{\bmu\in\Pad}\hat{J}^{\ell,(i)}(\bmu)\quad\text{s.t.}\quad \mathfrak q^{(k)}(\bmu):=\frac{\Delta^{\ell,(i)}_{J}(\bmu)}{\hat{J}^{\ell,(i)}(\bmu)}\le\delta^{(i)}.
\end{equation}
Here and whenever some quantity depends on the iteration $i$, we show it in the superscript $(i)$, like the RB cost $\hat J^{\ell,(i)}$. The ratio $q^{(i)}$ quantifies the accuracy of the RB and is used to define the TR. As we have seen in \eqref{5.0:Delta_J}, the value $\Delta^{\ell,(i)}_J(\bmu)$ is dependent on $\sigma_\qrm$, which is evaluated as $\sigma_\qrm=\sigma_\qrm^{(i)}=E^m_\qrm(\bmu^{(i)})^2/E^\ell_\qrm(\bmu^{(i)})^2$ at each iteration.

Whether the solution $\tilde\bmu$ of \eqref{eq:TR_subproblem} is accepted as the next step of the optimization algorithm, is decided based on the error-aware sufficient decrease condition (EASDC) introduced in \cite[Formula~(3.9)]{QGVW17}, i.e., if
\begin{equation}
    \label{5.1:EASDC}
    \hat{J}^{\ell,(i+1)}(\bmu^{(i+1)})\le\hat{J}^{\ell,(i)}(\bmu^{(i)}_{AGC}),
\end{equation}
where $\bmu^{(i)}_{AGC}$ is the approximated generalized Cauchy (AGC) point, defined as the steepest descent method solution in the initial direction $-\nabla\hat{J}^{\ell,(i)}(\bmu^{(i)})$. Since $\hat J^{\ell,(i)}$ refers to the reduced model at iteration $k$, while $\hat J^{\ell,(i+1)}$ refers instead to the model after the $(k+1)$-th (eventual) enrichment, the condition is not effortless to verify, and instead other similar conditions are tested, cf. \cite{QGVW17,KMOSV21,BKMOSV22}.

If the candidate is rejected, we can expect the surrogate model to be not accurate enough, and hence enrich the RB models or reduce the TR radius. On the other hand, if it is accepted, we can even decide that there is no need to enrich the model, as explained in \cite{BKMOSV22}. The algorithm then stops when $\|\bmu^{(i+1)}-\mathcal{P}_\Pad(\bmu^{(i+1)}-\nabla\hat{J}^h(\bmu^{(i+1)}))\|_2\le\varepsilon_{tr}$, where $\varepsilon_{tr}=10^{-5}$ is the overall tolerance.

\begin{rem}
    Following \cite{BKMOSV22,QGVW17} one can prove the convergence of our optimization method by assuming additional hypotheses. In fact, we need uniform Lipschitz  continuity for the reduced cost $\hat J^{\ell,(k)}$, the mapping $\bmu\mapsto\mathfrak q^{(i)}$ has to be uniformly continuous, \eqref{eq:TR_subproblem} should admit at least one optimal solution satisfying \eqref{5.1:EASDC} and the gradient of the reduced cost functional has to fulfill a so-called \emph{Carter condition} \cite{Car91}.\hfill$\Diamond$
\end{rem}

%%-----------------------------
\subsection{Numerical experiments (cont'd)}
\label{Sec:5.2}
%------------------------------

We will now show the accuracy and the speed-up of the TR optimization algorithm using the hierarchical error estimation. Let $\Omega=(0,1)$, $T=2$, $\kappa_1(x)\equiv\kappa_2(x)\equiv1$, $y_\circ(x)\equiv5$ and $\Pad$ defined as in Section~\ref{Sec:4.2}. Furthermore, let $\alpha^J=10^5 $, $\lambda=10^{-7}$, $\bmu^{(0)}=\hat\bmu=(3,3,3,3)$ (i.e. the middle point of $\Pad$) and let $u(t)=-3\cdot\bm{1}_{[0,4/3)}(t)+3\cdot\bm{1}_{[4/3,2]}(t)$. As discretization dimensions we choose $n=201$ (linear Lagrangian elements on 200 spatial nodes) and $K=201$, and as data we use a virtual noisy measurement of the $q$ state variable corresponding to the hidden parameter $\bmu^*=(2,3,4,5)$, namely
\begin{equation*}
    \wrm^k_j = \qrm^k_j(\bmu^*) + \varepsilon^k_j, \qquad\text{where } \varepsilon^k_j\sim\mathcal{N}(0,\sigma_{\mathsf{d}}^2) \text{ for any } j=1,\dots,n,\ k=1,\dots,K,
\end{equation*}
where $\sigma_{\mathsf{d}}^2=10^{-3}$ is the simulated measurement variance.

\begin{figure}[b]
\centering
\begin{tabular}{rr}
    \includegraphics{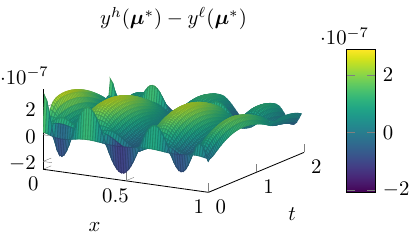}
    &
    \includegraphics{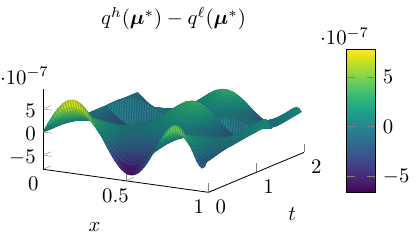}
\end{tabular}
\caption{\label{Fig:errors}Errors between FE and RB solutions for $\bmu^*=(2,3,4,5)$ and $u(t)=-3\cdot\bm{1}_{[0,4/3)}(t)+3\cdot\bm{1}_{[4/3,2]}(t)$ using linear Lagrangian elements.}
\end{figure}

Our computations are done in Python and the results of the TR optimization are compared with the results of the full-order (FO) optimization problem solved with the function \texttt{fmin\_l\_bfgs\_b} from the \texttt{scipy.optimize} library. The results are shown in Table~\ref{Tab_Results_1}. The speed-up of the TR optimization algorithm in the overall CPU time is 5.3, and we can see how the number of iterations and FO evaluations are much smaller. In the last two columns we can see the norms of the absolute and relative error between the true hidden parameter and the solution of the parameter optimization algorithm, i.e. $\mathsf{e}^{\mathsf{abs}}_\bmu:=\|\bmu^*-\bmu^\mathsf{opt}\|_2$ and $\mathsf{e}^{\mathsf{rel}}_\bmu:=\|\bmu^*-\bmu^\mathsf{opt}\|_2/\|\bmu^*\|_2$. As shown, both optimization results are accurate enough.
To show the accuracy of the RB approximation, in Figure~\ref{Fig:errors} the errors between the FE and RB approximations corresponding to the underlying hidden parameter $\bmu^*$ are plotted.

\begin{table}[ht]
\begin{center}
{\small\begin{tabular}{|l|c|c|c|c|c|} 
\hline
& Time & Iterations & FO evaluations & $\mathsf{e}^{\mathsf{abs}}_\bmu$ & $\mathsf{e}^{\mathsf{rel}}_\bmu$ \\\hline
FO optimization & 155 $s$ & 33 & 40 & 0.028 & 0.0039 \\
TR-RB & \phantom{1}29 $s$ & \phantom{3}3 & \phantom{3}5 & 0.029 & 0.0039 \\
\hline
\end{tabular}}
\end{center}
\caption{\label{Tab_Results_1}Results for $\bmu^*=(2,3,4,5)$ and $u(t)=-3\cdot\bm{1}_{[0,4/3)}(t)+3\cdot\bm{1}_{[4/3,2]}(t)$ using linear Lagrangian elements.}
\end{table}

The speed-up given by the TR optimization, shown in Table~\ref{Tab_Results_1} can be even higher when the speed-up given by the evaluation of the RB solution in comparison with the FE solution is higher. For example, using quadratic Lagrangian elements (so that $n=401$) the computational time for evaluating a FO solution is a bit higher (8 seconds) and the optimization algorithm takes longer. The results are shown in Table~\ref{Tab_Results_2}. Here the speed-up with respect to the overall CPU time is about 15, and it is reasonable to believe that for more accurate discretizations it can get even better.

\begin{table}[h]
\begin{center}
{\small\begin{tabular}{|l|c|c|c|c|c|} 
\hline
& Time & Iterations & FO evaluations & $\mathsf{e}^{\mathsf{abs}}_\bmu$ & $\mathsf{e}^{\mathsf{rel}}_\bmu$ \\\hline
FO optimization & 782 $s$ & 36 & 41 & 0.016 & 0.0022 \\
TR-RB & \phantom{1}52 $s$ & \phantom{3}3 & \phantom{2}3 & 0.018 & 0.0025 \\
\hline
\end{tabular}}
\end{center}
\caption{\label{Tab_Results_2}Results for $\bmu^*=(2,3,4,5)$ and $u(t)=-3\cdot\bm{1}_{[0,4/3)}(t)+3\cdot\bm{1}_{[4/3,2]}(t)$ using quadratic Lagrangian elements.}
\end{table}

As a further example, we simply change the hidden parameter to $\bmu^*=(4,4,2,1.5)$ and use linear Lagrangian elements. As we can see in Table~\ref{Tab_Results_3}, the results do not change significantly with respect to results seen in Table~\ref{Tab_Results_1}, both in the overall time and in number of iterations.

\begin{table}[ht]
\begin{center}
{\small\begin{tabular}{|l|c|c|c|c|c|} 
\hline
& Time & Iterations & FO evaluations & $\mathsf{e}^{\mathsf{abs}}_\bmu$ & $\mathsf{e}^{\mathsf{rel}}_\bmu$ \\\hline
FO optimization & 166 $s$ & 35 & 44 & 0.0082 & 0.0013 \\
TR-RB & \phantom{1}27 $s$ & \phantom{3}4 & \phantom{3}4 & 0.0082 & 0.0013 \\
\hline
\end{tabular}}
\end{center}
\caption{\label{Tab_Results_3}Results for $\bmu^*=(4,4,2,1.5)$ and $u(t)=-3\cdot\bm{1}_{[0,4/3)}(t)+3\cdot\bm{1}_{[4/3,2]}(t)$ using linear Lagrangian elements.}
\end{table}

On the other hand, if we change the input $u$ to, say, $u(t)=0.5\cdot\cos(10t) + 0.4\cdot\sin(20t)$, then the hidden parameter (in this example again $\bmu^*=(2,3,4,5)$) is not well approximated both for the FO and the reduced-order approximations. But nonetheless, the convergence time is much faster; cf. Table~\ref{Tab_Results_4}, and the optimization algorithm falls into a local minimum much faster.

\begin{table}[ht]
\begin{center}
{\small\begin{tabular}{|l|c|c|c|c|c|} 
\hline
& Time & Iterations & FO evaluations & $\mathsf{e}^{\mathsf{abs}}_\bmu$ & $\mathsf{e}^{\mathsf{rel}}_\bmu$ \\\hline
FO optimization & 45 $s$ & 12 & 13 & 1.00 & 0.13 \\
TR-RB & \phantom{1}9 $s$ & \phantom{3}2 & \phantom{3}3 & 1.00 & 0.13 \\
\hline
\end{tabular}}
\end{center}
\caption{\label{Tab_Results_4}Results for $\bmu^*=(2,3,4,5)$ and $u(t)=0.5\cdot\cos(10t) + 0.4\cdot\sin(20t)$ using linear Lagrangian elements.}
\end{table}

%------------------------------
\section{Conclusions}
\label{Sec:6}
%------------------------------

In this paper, we have considered a parameter-dependent coupled elliptic-parabolic problem. First, we have proved that a unique weak solution exists locally in time for some assumptions regarding the coupling non-linearity, the parameter space, the initial condition, and the control function (cf. Assumption~\ref{2.0:assumption_1}). Then, we have defined a full-order discretization, sufficiently accurate but expensive to solve. For this reason, the reduced basis (RB) approximation was defined.  Due to the non-linear nature of the problem, a rigorous, online-efficient error estimate is not computable. On the other hand, we have defined approximated hierarchical error estimators for the state variables and for a pretty general quadratic cost functional.

Through a weak greedy algorithm, we have built an RB model in the whole parameter space and found approximated error estimates. Later we have shown the accuracy and efficiency of such estimates on a random set of parameters. Consequently,  we have used the error estimates in a trust-region (TR) optimization framework, based on the recent work of RB-TR algorithms \cite{QGVW17,KMOSV21,BKMOSV22,PSV22}. The numerical results show a significant speed-up without sacrificing the accuracy of the optimized parameters.

Let us mention that our numerical experiments carried out in Section~\ref{Sec:5.2} show that the choice of the input function $u$ influences the parameter optimization with respect to identifiability. In \emph{optimal input design} the goal is to find the ``best'' input function for the parameter optimization; see, e.g., \cite{GP77,ADT07,KKBS04}. This is a possible future research direction; cf. \cite{PSV22}.

%%-----------------------------
\setcounter{section}{0}
\renewcommand*{\thesection}{\Alph{section}}
\section{Appendix}
\label{Sec:A}
%%-----------------------------

%%-----------------------------
\subsection{Proof of Theorem~\ref{2.2:thm_exist_C_est}}
\label{Sec:A.1}
%%-----------------------------

We show the existence of the solution of \eqref{2.2:elliptic_state_weak} by a fixed point argument, and the $\Cscrc$-estimate is obtained during the proof. For $\delta\in[0,1]$, f.a.a. $t\in[0,T]$, given $y\in\Yscr^T_M$, $u\in\Uad$ and $v\in\Cscrc$, let us consider the equation 
\begin{equation}
    \label{A.1:weak}
    \mu_3 \int_\Omega \kappa_2(x) q_x(t,x) \varphi'(x) \,\mathrm dx + \mu_4\delta\int_\Omega f(y(t,x),v(x))\varphi(x)\mathrm dx= u(t) \varphi(L)\quad\text{for all }\varphi\in\Vc.
\end{equation}
F.a.a. $t\in[0,T]$ we also introduce the solution operator
\begin{equation*}
    \mathcal T_t:[0,1] \times \Cscrc \to \Cscrc,
\end{equation*}
which maps $(\delta,v)$ to $q(t)$ as the solution of \eqref{A.1:weak} using $y(t)$ and $u(t)$. Then, using Leray-Schauder principle (cf., e.g., \cite[p.~189]{Fri64}) we prove that there exists a fixed point $q(t)$ of $\mathcal{T}_t(1,\cdot):\Cscrc\to\Cscrc$ f.a.a. $t\in[0,T]$.
\begin{description}
\item[Wellposedness of $\mathcal T_t$] Due to the fact that $y(t) \in C(\overline\Omega)$, $v\in C_\circ(\overline\Omega)$, and that $f$ is continuous, it can be easily be shown that $\mu_4 \delta f(y(t),v)\in C(\overline\Omega)$. Due to Assumption~\ref{2.0:assumption_1}-4) we can apply the Lax-Milgram theorem to derive the existence of a unique weak solution  $q(t)\in\Vc\hookrightarrow \Cscrc$ to \eqref{A.1:weak} satisfying
\begin{equation}
    {\|q(t)\|}_{C(\overline\Omega)}\le c_{\mathsf e}\|q(t)\|_\Vc\le c_1(M)\left(\|v\|_{C(\overline\Omega)}+|u(t)|\right),
\end{equation}
where the constant $c_1(M)$ does not depend on $v$, $y(t)$ or $u(t)$ but it does depend on $M$.
\item[Continuity of $\mathcal T_t(\delta,\cdot):\Cscrc\to\Cscrc$ for any $\delta\in\text{[0,1]}$ ] For any $v_1,v_2\in\Cscrc$ we set $q_1(t):=\mathcal T_t(\delta,v_1)$ and $q_2(t):=\mathcal T_t(\delta,v_2)$. Due to \eqref{2.1:c_embedding} and $\bmu\in\Pad$ we can write that 
\begin{align*}
    {\|q_1(t)-q_2(t)\|}_{C(\overline\Omega)}&\le c_{\mathsf e}{\|q_1(t)-q_2(t)\|}_\Vc\\
    &\le c_2\delta\left({\|f(y(t),v_1)-f(y(t),v_2)\|}_{C(\overline\Omega)} \right)\\
    &\le c_3(M)\left({\|\sinh(v_1)-\sinh(v_2)\|}_{C(\overline\Omega)}\right),
\end{align*}
where $c_2$ and $c_3(M)$ do not depend on $y(t)$, $v$ or $u(t)$. Thus the claims follow by the continuity of $\sinh$.
\item[Uniform continuity of $\mathcal T_t(\cdot\,,v):\text{[0,1]}\to\Cscrc$ on any bounded set  $\mathscr C \subset\Cscrc$] 
Let $\delta_1,\delta_2\in[0,1]$. Then for $q_1(t)=T(\delta_1\,,v)$ and $q_2(t)=T(\delta_1\,,v)$ we have
\begin{align*}
    \|q_1(t)-q_2(t)\|_{C(\overline\Omega)}
    &\le c_{\mathsf e}\|q_1(t)-q_2(t)\|_{\Vc}\\
    &\le c_4|\delta_1-\delta_2| \|f(y(t),v)\|_{L^{\infty}(\Omega)}\\
    &\le c_5(M)|\delta_1-\delta_2| \|\sinh(v)\|_{L^{\infty}(\Omega)}, 
\end{align*}
where $c_4$ and $c_5(M)$ do not depend on $y(t)$, $v$ or $u(t)$. Since $\|\sinh(v)\|_{C(\overline\Omega)}$ is uniformly bounded on the bounded set $\mathscr{C}\subset\Cscrc$, the operator $\mathcal{T}_t(\cdot\,,v):[0,1]\to\Cscrc$ is uniformly continuous on $\mathscr C$.
\item[Compactness of $\mathcal T_t( \delta, \cdot )$ for any $\delta \in \text{[0,1]}$] This follows from the fact that  for every $\delta \in  [0,1]$ the operator $\mathcal T_t(\delta,\cdot):\Cscrc\to\Vc$ is continuous and the space $\Vc$ is compactly embedded in $\Cscrc$ (see, e.g., \cite[Theorem 7.97]{Sal16}).
\item[Uniform boundedness of $q(t)$ satisfying $\mathcal T_t(\delta,q(t))=q(t)$ for any $\delta\in\text{[0,1]}$] Let an arbitrary $\delta \in [0,1]$ be given and let us test \eqref{A.1:weak} with $\varphi=q(t)$. Then, we get
\begin{equation*}
    \mu_3\int_\Omega\kappa_2(x) |q_x(t,x)|^2\mathrm dx+\mu_4\delta\int_\Omega f(y(t,x),q(t,x))q(t,x)\mathrm dx=q(t,L)u(t).
\end{equation*}
Using $y\in\Yscr^T_M$ and the fact that $\sinh(\qrm)\qrm\ge\qrm^2$ for any $\qrm\in\mathbb{R}$, we obtain that
\begin{equation*}
    \mu_{\mathsf a,3}\ka\,{\|q_x(t)\|}^2_H+\mu_4\delta\sqrt{1/M}{\|q(t)\|}^2_H\le|q(t,L)||u(t)|.
\end{equation*}
Hence, since $|q(t,L)|\le\|q(t)\|_\Vc$ and using \eqref{2.1:c_embedding}, Young's inequality, we get the estimate \eqref{2.2:q_est_C}, namely
\begin{equation*}
    {\|q(t)\|}_{\Vc}\le c(M)|u(t)|\le c(M)c_{\Uscr} ,
\end{equation*}
where $c(M)$ does not depend on $q$ or on $\delta$.
\item[Uniqueness of the solution of $\mathcal T_t(0,q(t))=q(t)$ in $\Cscrc$] This follows from the fact that for $\delta = 0$ problem \eqref{A.1:weak} is linear and well-posed.
\item[Application of Leray-Schauder principle] Having satisfied all of the hypotheses of the Leray-Schauder principle, the existence of a fixed point $T(1,q(t))=q(t)$ in $\Cscrc$ follows.
\item[Uniqueness] Consider $q_1(t)=\mathcal T_t(1,q_1(t))$, $q_2(t)=\mathcal T_t(1,q_2(t))$ and let $\bar{q}(t)=q_1(t)-q_2(t)$. From \eqref{A.1:weak} we can get
\begin{equation*}
    \mu_3 \int_\Omega \kappa_2\bar{q}_x(t) \varphi' \,\mathrm dx + \mu_4\int_\Omega \left(f(y(t),q_1(t))-f(y(t),q_2(t))\right)\varphi\mathrm dx=0\quad\text{for all }\varphi\in\Vc.
\end{equation*}
Choosing $\varphi=\bar{q}(t)$, we get
\begin{equation*}
    \mu_3 \int_\Omega \kappa_2(x) |\bar{q}_x(t)|^2\,\mathrm dx + \mu_4\int_\Omega \left(f(y(t),q_1(t))-f(y(t),q_2(t))\right)\bar{q}(t)\mathrm dx=0.
\end{equation*}
Using the mean value theorem ($f$ is continuously differentiable with respect to the variable q), we get
\begin{equation*}
    f(y(t,x),q_1(t,x))-f(y(t,x),q_2(t,x)) = \partial_\qrm f(y(t,x),\xi(t,x))\bar{q}(t,x)\qquad\text{f.a.a. }(t,x)\in Q_T
\end{equation*}
for some $\xi(t,x)$ between $q_1(t,x)$ and $q_2(t,x)$. It follows that
\begin{equation*}
    \mu_3 \int_\Omega \kappa_2(x) |\bar{q}_x(t)|^2\,\mathrm dx + \mu_4\int_\Omega \partial_\qrm f(y(t),\xi(t))|\bar{q}(t)|^2\mathrm dx=0\qquad\text{f.a.a. }(t)\in[0,T].
\end{equation*}
By Assumption~\ref{2.2:assumption_Y_M} the value $y(t,x)$ is positive in $Q_T$. Thus,
\begin{equation*}
    \partial_\qrm f(y(t),\xi(t))=\sqrt{y(t,x)}\cosh(\xi(t,x))\ge0\qquad\text{a.e. in }Q_T.
\end{equation*}
Utilizing $\kappa_2\ge\ka>0$ on $\Omega$ we infer that $\bar{q}_x=0$ a.e. in $Q_T$. Due to $\bar{q}\in\Vc$ f.a.a. $t\in[0,T]$ we have $\bar{q}=0$ a.e. in $Q_T$, which implies the uniqueness.
\end{description}

%%-----------------------------
\subsection{Proof of Theorem~\ref{thm:ExCoupledSystem}}
\label{Sec:A.2}
%%-----------------------------

The proof proceeds with the fixed point argument using Schauder's fixed point theorem \cite[Theorem 11.1 and Corollary 11.2]{GT83}. In this case, we consider the following linear parabolic equation for a given $v\in\Yscr^T_M$ and $q\in \Qscr^T$ f.a.a. $t\in[0,T]$
\begin{equation}
    \label{A.2:strong}
    \left\{
    \begin{aligned}
        y_t(t,x)-\mu_1\left(\kappa_1(x) y_x(t,x)\right)_x&=\mu_2 f(v(t,x),q(t,x))&&\text{f.a.a. }(t,x) \in Q_T, \\
        y_x(t,0)=y_x(t,L)&=0&&\text{f.a.a. }t\in(0,T],\\
        y(0,x)&= y_\circ(x)&&\text{f.a.a. }x\in \Omega.
    \end{aligned}
    \right.
\end{equation}
It is well known that for $w=(v,q)\in\Yscr^T_M\times \Qscr^T$,  there exists a unique \emph{weak solution} $y\in W(0,T;V,V')$ to \eqref{A.2:strong} satisfying $y(0)=y_\circ$ in $H$ and
\begin{align}
    \label{A.2:weak}
    \frac{\mathrm d}{\mathrm dt}\,{\langle y(t),\varphi\rangle}_H+\mu_1\hat{a}^1(y(t),\varphi)=\mu_2\hat g^1(w(t),\varphi)\quad\text{for all }\varphi\in V\text{ and f.a.a. }t\in(0,T].
\end{align}
Further,  for $y_\circ \in V$,  this solution is even more regular  and  it belongs to the space $W(0,T; E, H)$  with  
\[E:=\{u \in H^2(\Omega)=H^2(0,  L):  u_x(0)=u_x(L) =0 \}.\]
In this case,   due to the continuous embeddings  $E \hookrightarrow  V \hookrightarrow H $, we have $W(0,T; E, H) \hookrightarrow C([0,T];V)$ and with standard energy estimates  it can be shown that     
\begin{equation}
    \label{A.2:estimate_linear}
    \begin{split}
        \|y\|^2_{C([0,T];V)} + \| y \|^2_{L^2(0,T;H^2(\Omega))} \leq   c_1e^{c_2T} \left( \|y_\circ\|^2_V +          \|f(v,q)\|^2_{L^2(0,T;H)}, \right), 
    \end{split}
\end{equation}
where $c_1$ and $c_2$ depend only on $L$,   $\mu_1$,  and  $\mu_2$.   See,  e.g., \cite[p.~382]{Eva10}).  

 Now,   we define the mapping $\mathcal T: \Yscr_M^T \to W(0,T;E, H)$, where $y = \mathcal T(v)$ as  the solution to \eqref{A.2:weak} for any given  $v \in\Yscr_M^T$.
Next we show that the mapping $\mathcal{T}:\Yscr_M^{T_\circ}\to\Yscr_M^{T_\circ}$ for a suitable $T_\circ\in(0,T]$ is well-defined. Let $w=(v,q)\in\Yscr_M^T\times \Qscr^T$. Since $v\in\Yscr_M^T$ and $q\in \Qscr^T$, it follows that
\begin{align*}
    \mathfrak f(t,x) := \mu_2f(v(t,x),q(t,x))\in L^2(0,T;H)\simeq L^2(Q_T).
\end{align*}
We write  $y=\phi+\psi$, where $\phi$ and $\psi$ are the solution, respectively to
\begin{align}
    \left\{
    \begin{aligned}
       \phi_t(t,x)-\mu_1\left(\kappa_1(x)  \phi_x(t,x)\right)_x&=\mathfrak f(t,x)&&\text{f.a.a. }(t,x)\in Q_T,\\
        \phi_x(t,0)= \phi_x(t,L)&=0&&\text{f.a.a. }t\in(0,T],\\
    \phi(0,x)&= 0&&\text{f.a.a. }x\in \Omega.
    \end{aligned}
    \right.
\end{align}
and
\begin{align}
    \left\{
    \begin{aligned}
       \psi_t(t,x)-\mu_1\left(\kappa_1(x) \psi_x(t,x)\right)_x&=0&&\text{f.a.a. }(t,x)\in Q_T,\\
        \psi_x(t,0)=\psi_x(t,L)&=0&&\text{f.a.a. }t\in(0,T],\\
   \psi(0,x)&= y_\circ &&\text{f.a.a. }x\in \Omega.
    \end{aligned}
    \right.
\end{align}
Then, on the one hand for $\phi$ we can derive the estimate
\begin{align}
    \label{A.2:estimate_phi}
    {\|\phi\|}_{C([0,T];V)}+{\|\phi\|}_{L^2(0,T;H^2(\Omega))}\le c_1e^{c_2T}\,{\|\mathfrak f\|}_{L^2(0,T;H)}
\end{align}
with positive constants $c_1$ and $c_2$ depending only on $\Pad$, $\kappa$ and $\Omega$, while on the other hand using the fact that $y_\circ(x)\in[2/M,M/2]$ for every $x\in\Omega$, and the comparison principle \cite[A.1 Theorem and A.2 Corollary]{S89}, we can conclude that
\begin{equation}
    \label{A.2:estimate_psi}
    \psi(t,x)  \in \left[ \frac{2}{M}, \frac{M}{2} \right]  \quad  \text { f.a.a. } (t,x) \in [0,T]\times \Omega.
\end{equation}
Furthermore, for a given $T>0$, using $y=\psi+\phi$, \eqref{A.2:estimate_phi} and \eqref{A.2:estimate_psi}, we can write f.a.a. $(t,x)\in[0,T]\times\Omega$ that 
\begin{equation}
    \label{A.2:proof_y_in_YM_1}
    \begin{split}
        y(t,x) & \leq \frac{M}{2}+\phi(t,x) \leq \frac{M}{2}+ \hat{c}e^{c_2T}\,{\|\mathfrak f\|}_{L^2(0,T;H)},\\
        y(t,x)&\geq \frac{2}{M}+\phi(t,x)\geq \frac{2}{M}-\hat{c}e^{c_2T}\,{\|\mathfrak f\|}_{L^2(0,T;H)},
    \end{split}
\end{equation}
where $\hat{c}$ depend only on $\Pad$, $\ka$ and $\Omega$.
Since the function $[0,T]\ni t \mapsto \gamma (t) = \hat{c}e^{c_2t}\,\|\mathfrak f\|_{L^2(0,t;H)}$ is positive, increasing in $t$, and satisfies $\gamma(0)=0$, we can find $T_\circ\in(0,T]$ such that
\begin{equation}
    \label{A.2:proof_y_in_YM_2}
    \hat c_1e^{c_2t}\,{\|\mathfrak f\|}_{L^2(0,t;H)}\le\min\left\{\frac{M}{2},\frac{1}{M}\right\}\quad\text{for all }t \in [0,T_\circ].
\end{equation}
Finally,  using \eqref{A.2:proof_y_in_YM_1} and  \eqref{A.2:proof_y_in_YM_2} and applying Remark~\ref{Rem:y0}, we get
\begin{align*}
    y(t,x)&\le y_\circ(x)+\frac{M}{2}\le\frac{M}{2}+ \frac{M}{2}=M&&\text{for all }(t,x)\in[0,T_\circ]\times\Omega,\\
    y(t,x)&\ge y_\circ(x)-\frac{1}{M}\ge\frac{2}{M}-\frac{1}{M}=\frac{1}{M}&&\text{for all }(t,x)\in[0,T_\circ]\times\Omega,
\end{align*}
which implies that $y\in\Yscr_M^{T_\circ}$ holds.

Next, we show that the mapping $\mathcal T$ is continuous. Beforehand, we show that for given $v_1, v_2 \in \Yscr_M^T$, there exists $\gamma\ge 0$ such that
\begin{equation}
\label{A.2:Lip}
    {\| f(v_1,q)-f(v_2,q)\|}_{L^\infty(0,T;H)}\le\gamma\,{\| v_1 - v_2 \|}_{L^\infty(0,T;H)}.
\end{equation}
holds. We observe that f.a.a. $t\in [0,T]$
\begin{equation}
    \label{A.2:Lip2}
    \begin{split}
        {\|  f(v_1(t),q(t))- f(v_2(t),q(t))\|}^2_H &=  \int_\Omega\left|f(v_1(t,x),q(t,x))- f(v_2(t,x), q(t,x))\right|^2\,\mathrm dx\\
        &=\int_\Omega\left| ( \sqrt{v_1(t,x)} - \sqrt{v_2(t,x)} ) \sinh( q(t,x) )\right|^2\,\mathrm dx,
    \end{split}
\end{equation}
From \eqref{2.1:c_embedding} and \eqref{2.2:q_est_C} it follows that
\begin{align*}
    \max_{x\in\overline\Omega}|\sinh(q(t,x))|\le\sinh\left({\|q(t)\|}_{C(\overline\Omega)}\right)\le\sinh\left(c_{\mathsf e}\,{\|q(t)\|}_\Vc\right)\le c_3(M)\quad\text{f.a.a. }t\in[0,T],
\end{align*}
for a constant $c_3(M)>0$ independent of time. Thus, using \eqref{A.2:Lip2}, we can write f.a.a. $t\in[0,T]$ that
\begin{align*}
    {\|  f(v_1(t),q(t))- f(v_2(t),q(t))\|}^2_H &\le  c_3(M)^2\int_\Omega\left|\sqrt{v_1(t,x)} - \sqrt{v_2(t,x)}\right|^2\,\mathrm dx \\
    &=c_3(M)^2\int_\Omega \left|\frac{v_1(t,x) - v_2(t,x)}{\sqrt{v_1(t,x)} + \sqrt{v_2(t,x)}} \right|^2 \,\mathrm dx\\
    & \le \gamma(M) \int_{\Omega} \,|v_1(t,x)-v_2(t,x)|^2 \,\mathrm dx.
\end{align*}
Therefore, \eqref{A.2:Lip} holds for a constant $\gamma(M)>0$ independent of $v_1,v_2$ and $T$.  
 Using the similar estimate as in \eqref{A.2:estimate_linear} and  \eqref{A.2:Lip}, we can write for every $v_1, v_2 \in \Yscr_M^T$ that 
\begin{equation}
 \|  \mathcal T(v_1)-\mathcal T(v_2) \|^2_{C([0,T];V)} \leq c_1e^{c_2T} \left(\|f(v_1,q)-f(v_2,q)\|^2_{L^2(0,T;H)} \right) \leq  T \gamma^2 c_1e^{c_2T} \left(\|v_1-v_2\|^2_{L^{\infty}(0,T;H)} \right).
\end{equation}
Together with  \eqref{2.1:c_embedding} we get
    \begin{align*}
        {\| \mathcal T(v_1)-\mathcal T(v_2)\|}^2_{C([0,T];C(\overline\Omega))}&\le c_{\mathsf e}^2\, \|  \mathcal T(v_1)-\mathcal T(v_2) \|^2_{C([0,T];V)}\le c_{\mathsf e}^2T \gamma^2 c_1e^{c_2T} \left(\|v_1-v_2\|^2_{L^{\infty}(0,T;H)} \right)\\
        &\le\underbrace{\left[ L c_{\mathsf e}^2T \gamma^2 c_1e^{c_2T} \right]}_{=: \varrho(T)}\,{\|v_1-v_2\|}^2_{C(\overline Q_T)}=\varrho(T)\,{\|v_1-v_2\|}^2_{C(\overline Q_T)}.
    \end{align*}
Thus, we are shown that $\mathcal T$ is continuous.

In order to be able to use Schauder's fixed point theorem, it remains only to show that $\mathcal T$ is compact. To show this, we use the fact the solution of \eqref{A.2:strong} belongs to $W(0,T;E,H)$. Due to \cite[Thoerem 5.2]{A00}, this space is compactly embedded in $C([0,T];H^{2\theta}(\Omega))$ for $0<\theta<\frac{1}{2}$. Further, invoking \cite[Proposition 4.3]{T11},   $C([0,T];H^{2\theta}(\Omega))$ is  continuously embedded in $C(\overline Q_T)$ for $\theta > \frac{1}4$. Thus, choosing $\theta \in (\frac{1}{4} , \frac{1}{2})$, we can infer that $W(0,T;E,H)$ is compactly embedded in $C(\overline Q_T)$ and, thus, $\mathcal T$ is compact.

\noindent
\begin{acknowledgement}
\noindent
\textbf{Acknowledgement}.
We would like to thank Tim Keil (University of M\"unster, Germany) for fruitful discussions regarding the adaptive TR method.
\end{acknowledgement}

%%-----------------------------
%%      your bibliography
%%-----------------------------
\bibliographystyle{plain}
\bibliography{Bibliography}

\end{document}